\documentclass[12pt,leqno]{amsart}
\usepackage{a4wide}  

\usepackage{amsthm,amscd,amssymb,amsfonts}
\usepackage{latexsym}
\usepackage{txfonts}
\usepackage{exscale}
\usepackage{mathrsfs}
\usepackage{enumerate} 
\usepackage{stmaryrd}  
\usepackage{color}

\usepackage[colorlinks,citecolor=blue,hypertexnames=false]{hyperref}

\usepackage{appendix}

\definecolor{br}{RGB}{230,115,0}

\allowdisplaybreaks

\numberwithin{equation}{section}

\theoremstyle{plain}
\newtheorem{theorem}[equation]{Theorem}
\newtheorem{prop}[equation]{Proposition}
\newtheorem{corollary}[equation]{Corollary}
\newtheorem{lemma}[equation]{Lemma}

\theoremstyle{definition}
\newtheorem{defn}[equation]{Definition}

\theoremstyle{remark}

\newtheorem{remark}[equation]{Remark}

\numberwithin{equation}{section}

\renewcommand{\emptyset}{\mbox{\textup{\O}}}

\DeclareMathOperator{\dist}{dist}
\DeclareMathOperator{\supp}{spt}
\DeclareMathOperator{\Dir}{Dir}
\DeclareMathOperator{\card}{card}
\DeclareMathOperator{\AC}{AC}
\DeclareMathOperator{\Ng}{Ng}

\newcommand{\RR}{\mathbb{R}}
\newcommand{\DD}{\mathbb{D}}
\newcommand{\NN}{\mathbb{N}}

\newcommand{\GG}{\mathcal{G}}
\newcommand{\So}{{\mathbb{S}^1}}

\newcommand{\AQ}{\mathcal{A}_Q}
\newcommand{\AQt}{\mathcal{A}_Q^\pm}
\newcommand{\bxi}{\boldsymbol{\xi}}
\newcommand{\brho}{\boldsymbol{\rho}}
\newcommand{\bdeta}{\boldsymbol{\eta}}
\renewcommand{\iint}{{\int\!\!\!\!\! \int\!\!}}

\def\XXint#1#2#3{{\setbox0=\hbox{$#1{#2#3}{\int}$}
     \vcenter{\hbox{$#2#3$}}\kern-.5\wd0}}

\DeclareMathOperator{\spt}{spt}

\DeclareMathOperator{\Imag}{Im}

\newcommand\genus{{\textsl{g}}}
\def\a#1{\left\llbracket{#1}\right\rrbracket}
\newcommand\R{{\mathbb R}}

\reversemarginpar

\author[De Lellis]{Camillo De Lellis}
\address{School of Mathematics, Institute for Advanced Study, 1 Einstein Dr., Princeton NJ 05840, USA\\
and Universit\"at Z\"urich}
\email{camillo.delellis@math.ias.edu}

\author[Zhao]{Zihui Zhao}
\address{School of Mathematics, Institute for Advanced Study, 1 Einstein Dr., Princeton NJ 05840, USA}
\email{zzhao@ias.edu}

\begin{document}
\allowdisplaybreaks

\title{Dirichlet energy-minimizers with analytic boundary}

\thanks{The second author gratefully acknowledges support from the Institute for Advanced Study.}


\begin{abstract}
In this paper, we consider multi-valued graphs with a prescribed real analytic interface that minimize the Dirichlet energy. Such objects arise as a linearized model of area minimizing currents with real analytic boundaries and our main result is that their singular set is discrete in 2 dimensions. This confirms (and provides a first step to) a conjecture by B. White \cite{White97} that area minimizing $2$-dimensional currents with real analytic boundaries have a finite number of singularities.
We also show that, in any dimension, Dirichlet energy-minimizers with a $C^1$ boundary interface are H\"older continuous at the interface.
\end{abstract}

\maketitle

\section{Introduction and main result}

Consider a smooth closed curve $\Gamma$ in $\mathbb R^{2+n}$. The existence of oriented surfaces which bound $\Gamma$ and minimize the area can be approached in two different ways. Following the classical work of Douglas and Rado we can fix an abstract connected smooth surface $\Sigma_\genus$ of genus $\genus$ whose boundary $\partial \Sigma_\genus$ consists of a single connected component and look at smooth maps $\Phi: \Sigma_g \to \mathbb R^{2+n}$ with the property that the restriction of $\Phi$ to $\partial \Sigma_\genus$ is an homeomorphism onto $\Gamma$. We then consider the infimum $A_\genus (\Gamma)$ over all such maps $\Phi$ and all smooth Riemannian metrics $h$ on $\Sigma$ of the energy  
\[
\int_{\Sigma_\genus} |\nabla \Phi|^2{\rm dvol}_h \, .
\]
If $A_\genus (\Gamma) < A_{\genus -1} (\Gamma)$, then there is a minimizer, cf. \cite{Douglas,Courant40}, whose image is an immersed surface of genus $\genus$, with possible branch points. 
A different, more intrisic, approach was pioneered by De Giorgi, cf. \cite{DG}, in the codimension $1$ case, and by Federer and Fleming in higher codimension, cf. \cite{FF}. The latter looks at integral currents $T$ (a suitable measure-theoretic generalization of classical oriented submanifolds with boundary) whose boundary is given by $\a{\Gamma}$ and minimizes their mass, a suitable measure-theoretic generalization of the volume of classical submanifolds. The minimizer then always exists via the direct methods of the calculus of variations. 

There is a very natural question relating the two approaches: is every minimizer $T$ found by the Federer-Fleming theory a classical minimal surface with finite topology, namely a parametrized surface of some genus $\genus$? Note that if this were the case, then the sequence $\{A_\genus (\Gamma)\}_{\genus \in \mathbb N}$ would become constant for sufficiently large $\genus$.
When the codimension $n$ equals $1$ and $\Gamma$ is of class $C^{2,\alpha}$ for some $\alpha >0$, the interior regularity theorem of De Giorgi in \cite{De-GiorgiColombiniPiccinini72} and the boundary regularity theorem of Hardt and Simon in \cite{HS} imply that every minimizer $T$ is in fact a $C^{2,\alpha}$ embedded surface up to the boundary. Thus $T$ has finite genus $\genus_0$ and any conformal parametrization $\Phi$ over an abstract Riemann surface $\Sigma_{\genus_0}$ gives a minimizer in the sense of Douglas and Rado. On the other hand, Fleming in \cite{Fleming} showed a closed embedded curve $\Gamma$ in $\mathbb R^3$ of finite length for which $\{A_\genus (\Gamma)\}_{\genus \in \mathbb N}$ is not asymptotically constant.

The question is much more subtle in higher codimension, because singularities might arise, both at the interior and at the boundary.
In the work \cite{White97} White asks whether the topology of the minimizer $T$ is finite when $\Gamma$ is real analytic. If this conjecture were true, then $T$ would have finitely many singularities by the main theorem of \cite{White97}. The aim of this paper is to start a sort of reverse program to White's: under the assumption of real analyticity for the boundary $\Gamma$ we wish to show first that the set of boundary and interior singular points of $T$ is finite and hence to analyze the singularities and conclude that the topology of the minimizer is finite. 

It has been shown by Chang in \cite{Chang} that $T$ is smooth in $\R^n\setminus \Gamma$ up to a discrete set of singular branch points and in sufficiently small neighborhoods of such singular points the resulting branched surface is topologically a disk. We in fact refer to \cite{DSS1, DSS2, DSS3, DSS4} for a complete proof, as Chang needs a suitable modification of the techniques of Almgren's monumental monograph \cite{Alm} to start his argument, and the former has been given in full details in \cite{DSS3}. In order to attack White's conjecture it suffices therefore to deal with boundary regularity. In fact, even for $\Gamma$ of class $C^{2,\alpha}$, under the assumption that $\Gamma$ lies in the boundary of a uniformly convex set, the boundary regularity theorem of Allard \cite{AllB} implies that any minimizer $T$ is smooth at $\Gamma$; the general problem is however very subtle. So far the best available result is given in \cite{DDHM} and shows that the set of boundary regular points is dense in $\Gamma$ when $\Gamma$ is of class $C^{3,\alpha}$ for $\alpha >0$. The work \cite{DDHM} gives also an example of a smooth curve in $\R^4$ for which there is a unique Federer-Fleming minimizer with a sequence of singularities accumulating to a boundary branch point. This example has been modified in \cite{DDH} to produce $C^\infty$ embedded curves in complete $C^\infty$ Riemannian $4$-dimensional manifolds for which there is a unique Federer-Fleming minimizer with infinite topology. In particular there is a strong contrast to the codimension $1$ case: the real analyticity assumption in White's conjecture is, in a certain sense, needed\footnote{The examples of \cite{DDH} are curves in smooth almost K\"ahler manifolds $(\mathbb R^4, g)$, whose smooth metrics can be taken arbitrarily close to the euclidean one. However it is currently not known whether such examples exist in the Euclidean space.}.

\subsection{Linearized model} The analysis of interior singularities of area minimizing currents was pioneered by Almgren's monumental work in \cite{Alm} in the early eighties and recently revisited from a modern perspective by the first author and Emanuele Spadaro in \cite{DS}. The work \cite{DDHM} gives an Almgren type theory at the boundary, whereas the works \cite{DSS1,DSS2,DSS3,DSS4,Spolaor,DHMS1,DHMS2} extend the interior theory to other objects (almost calibrated currents and area minimizing currents modulo $p$). The starting point of all these papers, an essential discovery of Almgren, is to analyze the singularities for a suitable ``linearized model''. The main purpose of the present paper is to state and prove the appropriate linearized counterpart of White's conjecture. 

First of all we recall the notation $\AQ (\R^n)$ for the set of unordered $Q$-tuples of $\R^n$, which we will regard as nonnegative atomic measures with integer coefficients and total mass $Q$, cf. \cite[Introduction]{DS} for the formal definition and for the standard complete metric $\mathcal{G}$ which we will use on it. For atoms we will use the notation $\a{P}$ and thus elements in $\AQ (\R^n)$ will be denoted by $\sum_i \a{P_i}$. In what follows we will often write $\AQ$ instead of $\AQ (\R^n)$. 
We recall that for Sobolev functions $f\in W^{1,2}(\Omega, \AQ)$ (cf. again \cite[Introduction]{DS}) we set
\[ |Df|^2 := \sum_{j=1}^m |\partial_j f|^2, \]
where
\begin{equation}\label{eq:repD}
	|\partial_j f| = \sup_{i\in \NN}| \partial_j \GG(f, T_i)| \quad \text{almost everywhere in } \Omega,
\end{equation}
and $\{T_i\}_{i\in \NN}$ is a countable dense subset of $\AQ$. While such abstract definition is very direct and useful to work with, the Dirichlet energy turns out to be the sum of the Dirichlet energies of the different sheets in all cases where the multifunction $f$ can be ``nicely decomposed''. In an appropriate sense this can be justified also for any Sobolev functions, the reader is again referred to \cite{DS} for the relevant details.  

We now recall the notion of interior regular points.  

\begin{defn}[Interior regular point, Definition 0.10 of \cite{DS}]\label{def:intsing}
	A function $f\in W^{1,2}(\Omega, \AQ)$ is regular at a point $x\in \Omega$ if there exists a neighborhood $B$ of $x$ and $Q$ analytic functions $f_i: B\to \RR^n$ such that
	\[ f(y) = \sum_i \llbracket f_i(y) \rrbracket \quad \text{ for almost every } y\in B, \]
	and either $f_i(y) \neq f_j(y)$ for every $y\in B$, or $f_i \equiv f_j$.
	The complement of interior regular points is called the set of interior singular points.
\end{defn}

The following theorem on the interior regularity of Dir-minimizers was proven in \cite{DS}, refining a previous fundamental result by Almgren in \cite{Alm}:

\begin{theorem}[Theorem 0.12 in \cite{DS}]\label{thm:intsing}
	Let $f \in W^{1,2}(\Omega,\AQ)$ be Dir-minimizing and $m=2$. Then the interior singular set of $f$ consists of isolated points.
\end{theorem}

We now come to the boundary counterpart, following the approach of \cite{DDHM}. 
Suppose a hypersurface $\gamma$ divides a connected open set $\Omega \subset \RR^m$ into two connected components $\Omega^+$ and $\Omega^-$. For any set $K\subset \Omega$ we will use the notation $K^\pm$ for $K \cap \Omega^\pm$. Moreover, in order to avoid confusion, in the rest of the paper we will use the double integral symbol to indicate integration over subsets of $\RR^m$ with respect to the Lebesgue measure, and the single integral symbol to indicate integration over subsets of the hypersurface $\gamma$ with respect to the usual Hausdorff 
$(m-1)$-dimensional measure.

\begin{defn}
	We say that the pair $f=(f^+, f^-)$ is a $\left(Q-\frac{1}{2}\right)$-map with interface $(\gamma, \varphi)$ of class $W^{1,2}$  if there is  some (classical) function $\varphi\in H^{1/2} (\gamma, \mathbb R^n)$ such that
\begin{enumerate}[(i)]
	\item $f^+ \in W^{1,2}(\Omega^+, \AQ)$ and $f^-\in W^{1,2}(\Omega^-, \mathcal{A}_{Q-1} )$;
	\item $f^+|_{\gamma} = f^-|_{\gamma} + \llbracket \varphi \rrbracket$.
\end{enumerate}
We refer to \cite{DS,DDHM} for the trace theorems which allow to make sense of (ii) under our assumptions. For the corresponding set of pairs we will use the shorthand notation $W^{1,2} (\Omega, \AQt)$ and for each $f = (f^+, f^-) \in W^{1,2} (\Omega, \AQt)$ we define its Dirichlet energy as
\[ \Dir(f, \Omega) := \Dir(f^+, \Omega^+) + \Dir(f^-, \Omega^-) = \iint_{\Omega^+ }|Df^+|^2 + \iint_{\Omega^-} |Df^-|^2. \]

Finall, we say that $f=(f^+, f^-) \in W^{1,2}(\Omega, \AQt)$ is \textit{Dir-minimizing with interface $(\gamma,\varphi)$}, if $\Dir (g, \Omega) \geq \Dir (f, \Omega)$ for any other function $g \in W^{1,2}(\Omega, \AQt)$ with interface $(\gamma, \varphi)$ which agrees with $f$ outside of a compact set $K \subset \Omega$.
\end{defn}

The goal of the paper is to show that when the interface $(\gamma, \varphi)$ is real analytic and the domain is $2$-dimensional, Dir-minimizers enjoy a regularity theorem which is analogous to Theorem \ref{thm:intsing}. First of all 	a point $p\in \Omega\setminus \gamma$, namely belonging to either $\Omega^+$ or $\Omega^-$, will be called regular if it is a regular point for, respectively,$f^+$ or $f^-$ (cf. Definition \ref{def:intsing}). Its complement in $\Omega\setminus \gamma$ is the set of interior singular points, denoted by $\Sigma_f^i$. It remains to define regular points at the interface $\gamma$. 

\begin{defn}[Boundary regular point, Definition 2.6 of \cite{DDHM}]
	Let $f=(f^+, f^-)$ be a map in $W^{1,2}(\Omega, \AQt)$ with interface $(\gamma, \varphi)$. A point $p\in \gamma$ is regular if there are a ball $B_r(p)$, $(Q-1)$-functions $u_1, \cdots, u_{Q-1}: B_r(p) \to \RR^n$ and a function $u_Q: B_r^+(p) \to \RR^n$ such that
	\begin{itemize}
		\item $f^+ = \sum_{i=1}^Q \llbracket u_i \rrbracket$ on $B_r^+(p)$ and $f^- = \sum_{i=1}^{Q-1} \llbracket u_i \rrbracket$ on $B_r^-(p)$;
		\item For any pair $i, j \in \{1, \cdots, Q-1\}$ either the graphs of $u_i$ and $u_j$ are disjoint or they completely coincide;
		\item For any $i\in \{1, \cdots, Q-1\}$ either the graphs of $u_i$ and $u_Q$ are disjoint in $B_r^+(p)$ or the graph of $u_Q$ is contained in that of $u_i$.
	\end{itemize}
	The complement in $\gamma$ of the set of regular points is called the set of boundary singular points, denoted by $\Sigma_f^b$.
	\end{defn}

We can now state our main theorem:

\begin{theorem}\label{thm:main}
	Let $\Omega \subset \mathbb R^2$ and $(\gamma, \varphi)$ be an interface for which both $\gamma$ and $\varphi$ are real analytic. If $f\in W^{1,2} (\Omega, \AQt)$ is Dir-minimizing with interface $(\gamma, \varphi)$, then the singular set $\Sigma_f = \Sigma_f^i \cup \Sigma_f^b$ is discrete.
\end{theorem}

In passing, we need a suitable estimate on the H\"older continuity of minimizers at the interface $\gamma$. The latter result is however not confined to the special dimension $m=2$ nor to real analytic interfaces $(\gamma, \varphi)$ and, although it is not immediately relevant for our main purposes, we state it in a more general case in the following

\begin{theorem}\label{thm:main2}
	Let $m\in \mathbb N \setminus \{0,1\}$ and suppose $(f^+, f^-)$ is a Dir-minimizing $(Q-\frac{1}{2})$-map in $\Omega\subset \R^m$ with interface $(\gamma, \varphi)$ of class $C^1$. Then $(f^+, f^-)$ is H\"older regular.
\end{theorem}

In fact it is possible to give a precise estimate on a suitable H\"older seminorm of $f^{\pm}$ in terms of the regularity of the interface $(\gamma, \varphi)$ and the Dirichlet energy of the minimizer. For the precise statement we refer to Theorem \ref{thm:Holder}.

\subsection{Plan of the paper} The remaining sections are organized as follows. First of all in Section \ref{s:preliminaries} we make some preliminary elementary considerations on planar minimizers which will be particularly useful in the planar case of Theorem \ref{thm:main2} and in Theorem \ref{thm:main}. In Section \ref{s:hoelder} we address the general H\"older regularity result and prove therefore Theorem \ref{thm:main2}.  In the subsequent Section \ref{s:frequency} we give the fundamental computations leading to the monotonicity of the frequency function, a celebrated result of Almgren away from interface, extended at general interfaces in \cite{DDHM}: in our case the computations are simpler than in \cite{DDHM} because we can ``straighten the boudary'' using complex analysis. In Section \ref{s:rate} we use the frequency function estimate and the H\"older regularity to prove the existence of suitable blow-ups, or tangent functions, at singular points. A suitable modification of the argument given in \cite{DS} (which in turn borrowed from key ideas in \cite{Chang}) shows then the uniqueness of such objects. In Section \ref{s:classification} we give a list of necessary conditions that tangent functions must satisfy, which in turn leads to a suitable decomposition of them in simpler pieces (which we call irreducible maps). Such decomposition is combined together with the rate of convergence proven in Section \ref{s:rate} in order to decompose general Dir-minimizers at boundary singular points: the latter fact is then used in the final Section \ref{s:final} to conclude the proof of Theorem \ref{thm:main}.

\section{Reduction and preliminaries for the planar case}\label{s:preliminaries}

%
%
%

\subsection{Reduction of Theorem \ref{thm:main}}\label{s:mod_out_average} In this section we use elementary considerations in complex analysis to reduce Theorem \ref{thm:main} to a much simpler case. In order to state our theorem, we recall the definition of the map $\bdeta: \AQ (\R^n)\to \R^n$ which gives the barycenter of the atomic measure $T$:
\[
\bdeta \left(\sum_{i=1}^Q \a{P_i}\right) = \frac{1}{Q} \sum_{i=1}^Q P_i\, .
\]
In particular, if $(f^+, f^-)\in W^{1,2} (\Omega, \AQt)$ we can define two maps $(\eta^+, \eta^-)$ which are, respectively, the center of mass of the maps $f^+$ and $f^-$. In particular $\eta^\pm := \bdeta \circ f^\pm$,
where we make a slight abuse of notation because we keep the same symbol $\bdeta$ for two different maps, one defined on $\AQ$ and the other on $\mathcal{A}_{Q-1}$. Specifically:
\[
\eta^+ (x) = \frac{1}{Q} \sum_{i=1}^Q f^+_i (x)\qquad\mbox{and}\qquad
\eta^- (x) = \frac{1}{Q-1} \sum_{i=1}^{Q-1} f^-_i (x)\, .
\]
Theorem \ref{thm:main} can then be reduced to the following particular case:

\begin{theorem}\label{thm:main_simple}
Let $m=2$ and assume $(f^+,f^-)$ is Dir-minimizing in the unit disk $\mathbb D$ with interface $(\gamma, 0)$, where $\gamma$ is the coordinate axis $\{(x_1,0):x_1\in \mathbb R\}$. Assume further that $Q \eta^+ = (Q-1) \eta^-$. Then the singular set $\Sigma_f$ is discrete. 
\end{theorem}

From now on, we introduce the convention that, if $\gamma = \{(x_1, 0): x_1\in \mathbb R\}$, then the interface $(\gamma, \varphi)$ is denoted by $(\RR, \varphi)$. This is motivated by the fact that we will often identify $\RR^2$ with the complex plane $\mathbb C$, via $(x_1, x_2) \mapsto x_1+ix_2$. The set $\{x_2=0\}$ is then the real axis of $\mathbb C$ after such identification. 
The above theorem will be proved at the end of the paper. In the next paragraph we show how the general case of Theorem \ref{thm:main} follows from it. 

Assume $(f^+,f^-)$ is as in Theorem \ref{thm:main}. First of all observe that, if $\Sigma_f$ is not discrete, then by Theorem \ref{thm:intsing} $\Sigma_f$ must have an accumulation point $p\in \gamma$.
Modulo translation we may assume $p$ is the origin. Since $\gamma$ is analytic, we may choose a coordinate system so that the tangent to $\gamma$ satisfies $T_0 \gamma = \{ x_2 = 0\} = \RR$. In particular $\gamma$ must be (locally) the graph $\{(t, \zeta (t))\}$ of a function $\zeta (t)$ whose Taylor series at the origin is $\sum_{k\geq 2} \alpha_k t^k$ (where $\alpha_k = \frac{\zeta^{(k)} (0)}{k!} \in \RR$). Identify $\mathbb R^2$ with the complex plane and consider, in a neighborhood of the origin, the holomorphic map $\Phi$ given by $\Phi(z) = z+ \sum_{k\geq 2} i\alpha_k z^k$. By the inverse function theorem the latter map is invertible in a sufficiently small neighborhood $U$ of the origin (which can be assumed to be a disk) and its inverse over $\Phi (U)$ is also holomorphic. Since $\Phi$ is conformal, $(f^+\circ \Phi^{-1}, f^-\circ \Phi^{-1})$ is clearly a minimizer in 
$V := \Phi (U)$ and the interface is $(T_0 \gamma, \varphi \circ \Phi^{-1})$. Moreover $\Phi$ maps the segment $\{\Imag z = 0\}\cap U$ onto $\gamma$. 
We can thus assume, without loss of generality, that $\gamma = \mathbb R$. 

Next, since $\varphi$ is real analytic, by the Cauchy-Kowalevski Theorem $\varphi$ has a harmonic extension in a neighborhood of the origin, still denoted by $\varphi$. We then replace $f=(f^+, f^-)$ with
\[ f^+(x) \mapsto g^+ (x) := \sum_{i=1}^Q \llbracket f_i^+(x) - \varphi(x) \rrbracket , \quad f^-(x) \mapsto g^- (x) := \sum_{i=1}^{Q-1} \llbracket f_i^-(x) - \varphi(x) \rrbracket. \]
Indeed, given a map $(\bar{g}^+, \bar{g}^-)$ with interface $(\RR, 0)$ and same trace on $\partial \mathbb D$ as $(\bar g^+, \bar g^-)$, consider the corresponding map $(h^+, h^-)$ where we add $\varphi$ on each side. The latter has interface $(\RR, \varphi)$ and coincides with $(f^+, f^-)$ on $\partial \mathbb D$. Moreover we compute
\begin{align*}
\iint_{\mathbb D^+} |Dh^+|^2 &= \iint_{\mathbb D^+} |D \bar g^+|^2 +Q \iint_{\mathbb D^+} |D \varphi|^2 + 2 \underbrace{Q \iint_{\mathbb D^+} D\bdeta \circ \bar g^+ : D \varphi}_{=: I^+}\\
\, \\
\iint_{\mathbb D^-} |Dh^-|^2 &= \iint_{\mathbb D^-} |D \bar g^-|^2 +(Q-1) \iint_{\mathbb D^-} |D \varphi|^2 + 2\underbrace{(Q-1) \iint_{\mathbb D^-} D\bdeta\circ \bar g^- : D \varphi}_{=:I^-}
\end{align*}
Using that the function $\varphi$ is harmonic we compute
\begin{align*}
I^+ 
&= \underbrace{Q \int_{(\partial \mathbb D)^+} \bdeta \circ \bar g^+ \cdot \frac{\partial \varphi}{\partial \nu}}_{=:J^+}
- \underbrace{Q \int_{\RR\cap \mathbb D} \bdeta \circ \bar g^+ \cdot \frac{\partial\varphi}{\partial x_2}}_{=:K^+}\\ 
I^- &= \underbrace{(Q-1) \int_{(\partial D)^-} \bdeta \circ \bar g^- \cdot \frac{\partial \varphi}{\partial \nu}}_{=:J^-}  +  \underbrace{(Q-1) \int_{\RR\cap \mathbb D} \bdeta\circ \bar g^- \cdot \frac{\partial \varphi}{\partial x_2}}_{=:K^-}
\end{align*}
Observe that $J^+$ and $J^-$ are both independent of the choice of $(\bar g^+, \bar g^-)$, because the traces of the respective maps on $(\partial \mathbb D)^\pm$ equals those of $(g^+, g^-)$. On the other hand $Q \bdeta \circ \bar g^+ - (Q-1) \bdeta \circ \bar g^-= 0$ on $\RR\cap \DD$. Therefore $K^--K^+ =0$.  This implies that the difference 
\[
\iint_{\mathbb D^+} |Dh^+|^2 + \iint_{\mathbb D^-} |Dh^-|^2 - \iint_{\mathbb D^+} |D \bar g^+|^2 - \iint_{\mathbb D^-} |D\bar g^-|^2
\]
is actually a constant. In particular, if we could find a competitor for $(g^+, g^-)$ with lower energy, then we could transform it into a competitor for $(f^+, f^-)$ with lower energy: we conclude that $(g^+, g^-)$ must be a Dir minimizer with interface $(\RR, 0)$. 

Observe next that $\eta^+ = \bdeta \circ f^+$ and $\eta^- = \bdeta \circ f^-$ are harmonic functions in $\DD^+$ and $\DD^-$, respectively. For any $x=(x_1, x_2)\in \RR^2$, we denote $\bar x = (x_1, -x_2)$ the reflection point of $x$ across $\RR$. We define a function $\phi: \mathbb D \to \RR^n$ as 
\begin{equation}
	\phi(x) = \left\{ 
	\arraycolsep=1.4pt\def\arraystretch{2.2} 
	\begin{array}{ll}
		\dfrac{ Q\eta^+(x) - (Q-1) \eta^-(\bar x) }{ 2Q-1}, & \quad x_2 \geq 0, \\
		\dfrac{(Q-1) \eta^-(x) - Q \eta^+(\bar x)}{2Q-1}, & \quad x_2 \leq 0.
	\end{array} \right.
\end{equation}
Clearly $\phi$ is harmonic in $\mathbb D \setminus \RR$.
By the boundary condition $f^+|_{\gamma} = f^-|_{\gamma} + \llbracket 0 \rrbracket$, we know 
\[ Q\eta^+ = (Q-1) \eta^- \quad \text{ on } \RR. \]
Hence $\phi$ is continuous and odd in the variable $x_2$. In particular $\phi$ is harmonic on all of $\mathbb D$. Therefore by modifying $(f^+, f^-)$ as follows 
\[ f^+(x) \mapsto \widetilde f^+(x):= \sum_{i=1}^Q \llbracket f_i^+(x) - \phi(x) \rrbracket, \quad x\in \RR^m_+, \]
\[ f^-(x) \mapsto \widetilde f^-(x):= \sum_{i=1}^{Q-1} \llbracket f_i^-(x) - \phi(x) \rrbracket, \quad x\in \RR^m_- \]
and repeating the same computations as above we conclude that
the new function $(\widetilde f^+, \widetilde f^-)$ is still a Dir-minimizer with the same interface $(\RR, 0)$. Notice also that
\[
	\sum_{i=1}^Q \widetilde f^+_i(x) = Q \eta^+(x) - Q \phi(x) = \frac{Q(Q-1)}{2Q-1} \left( \eta^+(x) + \eta^-(\bar x) \right),
\]
\[
	\sum_{i=1}^{Q-1} \widetilde f^-_i(x) = (Q-1) \eta^-(x) - (Q-1) \phi(x) = \frac{Q(Q-1)}{2Q-1} \left( \eta^-(x) + \eta^+(\bar x) \right),
\]
and thus
\[
	\sum_{j=1}^Q \widetilde f^+_j(x) = \sum_{j=1}^{Q-1} \widetilde f^-_j(\bar x).
\] 
For simplicity we still denote the new function as $(f^+, f^-)$, except that its center of mass $(\eta^+, \eta^-)$ now enjoys an additional symmetry:
\begin{equation}\label{eq:avgsym}
	Q \eta^+(x) = (Q-1) \eta^-(\bar x).
\end{equation}
This symmetry is invariant under translation, scaling and uniform limit.

\subsection{Decomposition into irreducible maps} In this section we extend a suitable decomposition of $Q$-valued maps on the circle to the case of $(Q-\frac{1}{2})$-valued maps.
Recall that map $g\in W^{1,p}(\So, \AQ)$ is called \textit{irreducible} if there is no decomposition of $g$ into two simpler $W^{1,p}$ functions (cf. \cite{DS}), namely if there are no integers $Q_1, Q_2>0$ and maps $g_1\in W^{1,p} (\So, \mathcal{A}_{Q_1}), g_2 \in W^{1,p} (\So, \mathcal{A}_{Q_2})$ such that $g = g_1+g_2$ (in particular $Q_1+Q_2=Q$).

\begin{defn}[Irreducible $(Q-\frac{1}{2})$-maps on $\mathbb S^1$]
	A map $g = (g^+, g^-)\in W^{1,p}(\So, \AQt)$ with interface $(\RR, \varphi)$ is called \textit{irreducible} if there is no decomposition of $g$ into the ``sum'' of a map $g_1 \in W^{1,p}(\So,\mathcal{A}_{Q_1})$ and a map $g_2 \in W^{1,p}(\So, \mathcal{A}_{Q_2}^{\pm})$ with the same interface $(\gamma,\varphi)$, where the positive integers $Q_1, Q_2$ satisfy $Q_1 + Q_2 = Q$.
	The ``sum'' is understood in the following sense:
\begin{align*}
&g^+ = g_1 + g_2^+\qquad \mbox{on $(\mathbb S^1)^+ = \{z\in \mathbb C: |z|=1, \text{Re}\, z >0\}$}\\
&g^- = g_1 + g_2^-\qquad \mbox{on $(\mathbb S^1)^- = \{z\in \mathbb C: |z|=1, \text{Re}\, z <0\}$.} 
\end{align*}
\end{defn}
\begin{remark}
	By the above definition, clearly any function $g\in W^{1,p}([0,\pi], \RR^n)$ satisfying $g(0) = g(\pi) = 0$ is irreducible with $Q=1$ (the interface being $(\RR, 0)$).
\end{remark}

The decomposition of $W^{1,p}$ $(Q-\frac{1}{2})$-valued map on the circle is then a corollary of the following proposition for $Q$-valued maps, where, for any interval $I=[a,b]\subset \RR$, we denote by $\AC(I,\AQ)$ the space of absolutely continuous functions taking values in the metric space $(\AQ, \mathcal{G})$.

\begin{prop}[Proposition 1.2 of \cite{DS}]\label{prop:oddecomp}
	Let $g\in W^{1,p}(I, \AQ)$. Then
	\begin{enumerate}[(a)]
		\item $g\in \AC(I, \AQ)$ and moreover, $g\in C^{0,1-\frac1p}(I,\AQ)$ for $p>1$;
		\item There are $g_1, \cdots, g_Q\in W^{1,p}(I, \RR^n)$ s.t. $f =\sum_i \a{g_i}$ and $|Dg_i| \leq |Dg|$ a.e.
	\end{enumerate}
\end{prop}

\begin{prop}[Decomposition of $W^{1,p}(\mathbb{S}^1, \mathcal{A}_Q^{\pm})$]\label{prop:decomp}
	A map $g\in W^{1,p}(\So, \AQt)$ with interface $(\RR,\varphi)$ is either irreducible, or it can be decomposed as $g= g_0 + \sum_{j=1}^J g_j$, where $g_0\in W^{1,p}(\So, \mathcal{A}_{Q_0}^{\pm}) $ is irreducible with interface $(\RR, \varphi)$, and each $g_j \in W^{1,p}(\So, \mathcal{A}_{Q_j})$ is irreducible. Moreover, a map $g\in W^{1,2}(\So, \AQt)$ with interface $(\RR, \varphi)$ is irreducible if and only if the following two conditions are satisfied:
	\begin{enumerate}[(i)]
		\item $\card(g^+(\theta)) = Q$ for every $\theta\in [0,\pi]$, and $\card(g^-(\theta)) = Q-1$ for every $\theta \in [\pi,2\pi]$.
		\item There exists a $W^{1,p}$ map $\zeta:\So \to \RR^n$ with $\zeta(0) = \varphi (1)$ and $\zeta(2\pi) = \varphi (-1)$ such that $g$ unwinds to $\zeta$, in the following sense: $g^+ = \sum_{j=1}^Q \llbracket g_j^+ \rrbracket$ and $g^-= \sum_{j=1}^{Q-1} \llbracket g_j^- \rrbracket$ with
		\begin{equation}\label{eq:unwind1}
			g_j^+(\theta) = \zeta\left(\frac{2\theta}{2Q-1} + \frac{4\pi}{2Q-1} (j-1) \right), \quad \theta\in [0,\pi], j=1,\cdots,Q, 
		\end{equation} 
		\begin{equation}\label{eq:unwind2}
			g_j^-(\theta) = \zeta\left(\frac{2\theta}{2Q-1} + \frac{4\pi}{2Q-1} (j-1) \right), \quad \theta\in[\pi,2\pi], j=1,\cdots,Q-1. 
	\end{equation}
	\end{enumerate}
\end{prop}

\begin{proof}
	The existence of an irreducible decomposition in the above sense is an obvious consequence of the definition of irreducible maps. It remains to show the characterization of irreducible maps.
	
	By Proposition \ref{prop:oddecomp} a map satisfying (i) and (ii) is clearly irreducible with interface $(\gamma, \varphi)$. Suppose $g\in W^{1,p}(\So, \AQt)$ with interface $(\gamma, \varphi)$ is irreducible. Without loss of generality (i.e. after possible subtracting to all sheets an extension of $\varphi$) we can assume $\varphi \equiv 0$. Namely 
	\begin{equation}\label{eq:bcirred}
		g^+|_{\gamma} = g^-|_{\gamma} + \llbracket 0 \rrbracket. 
	\end{equation} 
	Recalling Proposition \ref{prop:oddecomp}, we consider  a selection $g^+_1, \cdots, g^+_Q \in W^{1,p}([0,\pi], \RR^n)$ of the map $g^+\in W^{1,p}([0,\pi], \AQ)$, and  a selection $g^-_1, \cdots, g^-_{Q-1} \in W^{1,p}([\pi, 2\pi], \RR^n)$ of $g^- \in W^{1,p}([\pi,2\pi], \mathcal{A}_{Q-1})$. We assume without loss of generality that $g_1^+(0) = \llbracket 0 \rrbracket$. By the boundary condition \eqref{eq:bcirred}, there exists an integer $Q_0$, $1\leq Q_0 \leq Q$, such that after reordering the selections $g^+_i(\pi) = g^-_{i}(\pi) \neq 0$ and $g_i^-(2\pi) = g_{i+1}^+(0) $ for all $i = 1, \cdots, Q_0-1$, $g^+_{Q_0}(\pi) = 0$. Suppose $Q_0 < Q$, then we define
	\[ f_1^+ = \sum_{i=1}^{Q_0} \llbracket g_i^+ \rrbracket, \quad f_1^- = \sum_{i=1}^{Q_0-1} \llbracket g_{i}^- \rrbracket, \]
	and
	\[ f_2 = \left\{ \begin{array}{ll} \sum\limits_{i=Q_0 +1 }^{Q} \llbracket g_i^+ \rrbracket, & \theta\in [0,\pi] \\
	\sum\limits_{i=Q_0}^{Q-1} \llbracket g_i^- \rrbracket, & \theta\in [\pi,2\pi]	
 		\end{array}\right.
 	\]
 	By \eqref{eq:bcirred}, the map $f_1:= (f_1^+, f_1^-)$ lies in $ W^{1,p}(\So, \mathcal{A}_{Q_0}^{\pm})$ with interface $(\gamma, \varphi)$; the map $f_2$ is well-defined on $\gamma$, i.e. $f_2(\pi-) = f_2 (\pi+)$ and $f_2(2\pi) = f_2(0)$, and moreover $f_2 \in W^{1,p}(\So, \mathcal{A}_{Q - Q_0})$. In other words, this gives a nontrivial decomposition of the irreducible map $g$, contradiction. Hence $Q_0 = Q$, and we define the function $\zeta$ by following $g_i^+, g_i^-$, $g_{i+1}^+$ in order.
 	
 	Suppose $\card(g^+) \neq Q$, that is, there exist $\theta_0 \in [0,\pi]$ and $i_1<i_2$ such that $g^+_{i_1}(\theta_0) = g^+_{i_2}(\theta_0)$. Let
 	\[ \widetilde g^+ = \left\{\begin{array}{ll}
 		g_{i_1}^+, & \theta \in [0,\theta_0] \\
 		g_{i_2}^+, & \theta \in [\theta_0, \pi].
 	\end{array} \right. \]
 	Then the following map gives a decomposition of $g$:
 	\[ f_1^+ = \sum_{i=1}^{i_1-1} \llbracket g_i^+ \rrbracket + \llbracket \widetilde g \rrbracket + \sum_{i=i_2+1}^{Q} \llbracket g_i^+ \rrbracket, \quad  f_1^- = \sum_{i=1}^{i_1-1} \llbracket g_i^- \rrbracket + \sum_{i=i_2}^{Q-1} \llbracket g_i^- \rrbracket. \]
 	Since $(f_1^+, f_1^-) \in W^{1,p}(\So, \mathcal{A}_{Q+i_1-i_2}^\pm)$ with interface $(\gamma, \varphi)$, and $Q+i_1-i_2 < Q$, this is a nontrivial decomposition of the irreducible map $g$, contradiction. Hence $\card(g^+) = Q$. Similarly $\card(g^-) = Q-1$.
\end{proof}

\subsection{Rolling and unrolling}
 The decomposition of the previous section can be used to construct efficient competitors to Dirichlet minimizers in the planar case. Again the situation is similar to that of $Q$-valued maps. Keeping our identification $\mathbb R^2 = \mathbb C$ we will denote by $[0,1]$ the ``slit'' $\{(x_1, 0): 0\leq x_1 \leq 1\}$ and on the domain $\DD\setminus [0,1]$ we will consider polar coordinates $(r, \theta) \in ]0,1[\times ]0, 2\pi[$, via the usual parametrization $(r, \theta)\mapsto r e^{i\theta}$. Given a map $\zeta\in W^{1,2} (\DD\setminus [0,1], \RR^n)$ we can define two maps $\zeta^u, \zeta^l \in H^{1/2} ([0,1], \RR^n)$ which are, respectively, the ``upper'' and ``lower'' traces of $\zeta$ on the slit $[0,1]$. In particular in polar coordinates we can naturally extend $\zeta$ to $]0,1[\times [0, 2\pi]$ setting $\zeta (r, 0) = \zeta^u (r)$ and to $\zeta (r, 2\pi) = \zeta^l (r)$. In the next lemma and its applications we will follow the latter convention.

\begin{lemma}[Unrolling, analogue of Lemma 3.12 in \cite{DS}]\label{lm:unrolling}
	Suppose $\zeta\in W^{1,2}(\DD\setminus [0,1], \RR^n)$ and consider the $\left(Q-\frac12 \right)$-valued function $f=(f^+, f^-)$ defined as follows:
	\begin{align}
		f_j^+(r,\theta) &= \zeta\left(r^{\frac{2}{2Q-1}}, \frac{2\theta}{2Q-1} + \frac{4\pi}{2Q-1} (j-1) \right), \quad \theta\in [0,\pi], j=1,\cdots,Q, \label{eq:unwindsld1}\\
		f_j^-(r,\theta) &= \zeta\left(r^{\frac{2}{2Q-1}}, \frac{2\theta}{2Q-1} + \frac{4\pi}{2Q-1} (j-1) \right), \quad \theta\in[\pi,2\pi], j=1,\cdots,Q-1. \label{eq:unwindsld2}
	\end{align} 
	(For $Q=1$ we just ignore $f^-$.)
	Then $f\in W^{1,2}(\DD, \AQt)$ and
	\begin{equation}\label{eq:unrollD}
		\Dir(f,\mathbb{D}) = \iint_{\mathbb{D}} |D\zeta|^2. 
	\end{equation} 
	Moreover, if $\zeta|_{\mathbb{S}^1} \in W^{1,2}(\mathbb{S}^1,\RR^n)$, then $f|_{\So} \in W^{1,2}(\So,\AQt)$ and
	\begin{equation}\label{eq:unrollH}
		\Dir(f|_\So,\So) = \frac{2}{2Q-1} \int_{\So} |\partial_\tau \zeta|^2,
	\end{equation}
	where $\partial_\tau$ denotes the tangential derivative on $\So$.
\end{lemma}

\begin{proof}
	We define the following subsets of the unit disk,
	\[ \mathcal{C} = \left\{re^{i\theta}: 0<r<1, \theta \neq 0\right\}, \]
	\[ \mathcal{C}^+ = \left\{re^{i\theta}: 0<r<1, 0< \theta <\pi \right\}, \quad \mathcal{C}^- = \left\{re^{i\theta}: 0<r<1, \pi < \theta <2\pi \right\}; \]
	\[ \mathcal{D}_j = \left\{re^{i\theta}: 0<r<1,\, \frac{2\pi}{2Q-1} 2(j-1) < \theta <  \frac{2\pi}{2Q-1} 2j \right\}, \quad j=1, \cdots, Q-1, \]
	\[ \mathcal{D}_j^+ = \left\{re^{i\theta}: 0<r<1,\, \frac{2\pi}{2Q-1} 2(j-1)< \theta <  \frac{2\pi}{2Q-1} (2j-1) \right\}, \quad j=1, \cdots,Q, \]
	\[ \mathcal{D}_j^- = \left\{re^{i\theta}: 0<r<1,\, \frac{2\pi}{2Q-1} (2j-1)< \theta <  \frac{2\pi}{2Q-1} 2j \right\}, \quad j=1, \cdots,Q-1. \]
	For $j=1, \cdots, Q-1$, we define $\varphi_j: \mathcal{C} \to \mathcal{D}_j$ as
	\[ \varphi_j(re^{i\theta}) = r^{\frac{2}{2Q-1}} e^{i\left( \frac{2\theta}{2Q-1} + \frac{4\pi}{2Q-1}(j-1) \right)}; \]
	and we define $\varphi_Q: \mathcal{C}^+ \to \mathcal{D}_Q^+$ as
	\[ \varphi_Q(re^{i\theta}) = r^{\frac{2}{2Q-1}} e^{i\left( \frac{2\theta}{2Q-1} + \frac{4\pi}{2Q-1}(Q-1) \right)}. \]
	Then 
	\[ f^+ = \sum_{j=1}^Q \llbracket \zeta \circ \varphi|_{\mathcal{C}^+} \rrbracket \text{ and } f^- = \sum_{j=1}^{Q-1} \llbracket \zeta \circ \varphi|_{\mathcal{C}^-} \rrbracket. \]
	Since $re^{i\theta} \mapsto r^{\frac{2}{2Q-1}} e^{i \frac{2\theta}{2Q-1}}$ is a conformal map, each $\varphi_j$ is conformal.
	So by the invariance of the Dirichlet energy under conformal mappings, we deduce that $f^{+} \in W^{1,2}(\mathcal{C}^+, \AQ)$, $f^- \in W^{1,2}(\mathcal{C}^-, \mathcal{A}_{Q-1})$ and
	\begin{align*}
		\Dir(f, \mathcal{C}) & = \Dir(f^+, \mathcal{C}^+) + \Dir(f^-, \mathcal{C}^-) 
		= \sum_{j=1}^Q \Dir(\zeta \circ \varphi_j, \mathcal{C}^+) + \sum_{j=1}^{Q-1} \Dir(\zeta \circ \varphi_j, \mathcal{C}^-) \\
		& = \sum_{j=1}^Q \Dir(\zeta , \mathcal{D}_j^+) + \sum_{j=1}^{Q-1} \Dir(\zeta , \mathcal{D}_j^-) = \Dir\left(\zeta, \cup_{j=1}^{Q-1} \mathcal{D}_j \cup \mathcal{D}_Q^+ \right) 
		= \iint_{\DD} |D\zeta|^2.
	\end{align*}
	On the other hand, since
	\[ \partial_\tau \left( \zeta \circ \varphi_j \right) = \partial_\theta \left(\zeta \circ \varphi_j \right) = \frac{2}{2Q-1} \partial_\tau \zeta \circ \varphi_j, \]
	we have
	\[ \Dir(\zeta\circ \varphi_j|_{\So}, (\So)^+) = \int_{(\So)^+} \left( \frac{2}{2Q-1} \right)^2 | \partial_\tau \zeta \circ \varphi_j|^2 = \frac{2}{2Q-1} \int_{\frac{2\pi}{2Q-1} 2(j-1) }^{\frac{2\pi}{2Q-1} (2j-1)} |\partial_\tau \zeta|^2.  \]
	An entirely analogous computations on $(\So)^-$ makes it straightforward to show that the restriction $f|_{\So}$ of $f$ to $\So$ belongs to  $W^{1,2}(\So, \AQt)$ and that
	\[ \Dir(f|_{\So}, \So) = \frac{2}{2Q-1} \int_{\So} |\partial_\tau \zeta|^2. \qedhere\]
\end{proof}

\section{H\"older continuity at the interface}\label{s:hoelder}

In this section we prove the H\"older regularity Theorem \ref{thm:main2}, whose conclusion we make more quantitative in the following statement. 

\begin{theorem}[Boundary H\"older regularity of Dir-minimizer, analogue of Theorem 3.9 in \cite{DS}]\label{thm:Holder}
	For every $0<\delta< \frac12$, there exist constant $\alpha = \alpha(m,Q)\in (0,1)$ and $C=C(m,n,Q, \delta)$ with the following property. Assume that $\gamma$ is a $C^1$ graph of a function $\zeta$ over $\mathbb R$ passing through the origin with $\|\zeta\|_{C^1}\leq 1$ and that $\varphi \in C^1 (\gamma)$. If $f\in W^{1,2}(B_1, \AQt)$ is Dir-minimizing with interface $(\gamma,\varphi)$, then
	\begin{align}
		[f]_{C^{0,\alpha}(\overline B_{\delta})} := &\; \max \left\{  \sup_{x,y\in \overline{B^+_{\delta}}}  \frac{\GG(f^+(x),f^+(y))}{|x-y|^\alpha}, \sup_{x,y\in \overline{B^-_{\delta}}}  \frac{\GG(f^-(x),f^-(y))}{|x-y|^\alpha} \right\}\nonumber\\
 \leq &\; C \Dir(f,B_1)^{\frac{1}{2}} + C \|D \varphi\|_{C^0}.
	\end{align}
\end{theorem}
The proof consists of two main steps. A comparison argument is used to prove a suitable decay of the Dirichlet energy on balls with vanishing radius. The decay is then combined with a Campanato-Morrey estimate to show H\"older regularity. 

\subsection{Campanato-Morrey estimate}
We first record the following extension of a classical result by Morrey. In the case of $Q$-valued maps we refer to \cite{DS}. In our case we need a suitable additional argument to treat the case of $\left(Q-\frac{1}{2} \right)$-valued functions.

\begin{lemma}[Campanato-Morrey estimate]\label{lm:CM}
	Suppose $(f^+,f^-)\in W^{1,2}(B_1, \AQt)$ is a map with interface $(\gamma, \varphi)$ as in Theorem \ref{thm:Holder}. If there exist $\beta\in (0,1]$ and $A\geq 0$ such that
	\begin{equation}
		\iint_{B_r(y)} |Df|^2 \leq A r^{m-2+2\beta} \quad \text{ for every } y\in B_1 \text{ and almost every } r\in (0,1-|y|),
	\end{equation}
	then for every $0<\delta<1$, there is a constant $C=C(m, \beta, \delta, \gamma)$ such that
	\[ [f]_{C^{0, \beta}(\overline B_\delta)} \leq C \sqrt{A} + C \delta^{1-\beta} \|D \varphi\|_{C^0}\, . \]
\end{lemma}

\begin{proof}
	We first extend $(f^+,f^-)$ to a function $g:B_1 \to \AQ(\RR^n)$ as follows. We use the $C^1$ regularity of $\gamma$ and $\varphi$ to extend $\varphi$ to a $C^1$ function $\phi$ over $B_1$ satisfying the estimate $\|D \phi\|_{C^0(B_1)}\leq C \|D \varphi\|_{C^0(\gamma)}$, where $C$ depends on $m$ and the $C^1$-norm of $\gamma$.
	\begin{equation}\label{eq:extension}
		g(x) := \left\{\begin{array}{ll}
		f^+(x), & x\in B_1^+ \cup \gamma, \\
		f^-(x) + \llbracket \phi (x) \rrbracket, & x\in B_1^-.
	\end{array}\right. 
	\end{equation} 
	Since $f^+|_{\gamma} = f^-|_{\gamma} + \llbracket \varphi \rrbracket$, the function $g$ belongs to $W^{1,2} (B_1, \AQ)$, by the trace theory of \cite{DS}. Moreover, the theory in \cite{DS} can be easily used to prove that 
\[ \iint_{B_r} |Dg|^2 = \iint_{B_r} |Df|^2 + \iint_{B_r^-} |D\phi|^2  \leq r^{m-2+2\beta} \left(A+ r^{2-2\beta} \|D \phi\|_{C^0}^2\right). \] 	
	By the Campanato-Morrey estimate for $Q$-valued functions (see \cite[Proposition 2.14]{DS}), we conclude that 
	\[ \sup_{x,y\in \overline{B_\delta}} \frac{\mathcal{G}(g(x),g(y))}{|x-y|^{\beta}} \leq C \left( \iint_{B_1} |Dg|^2 \right)^{\frac12}\, .\]
Since clearly $\mathcal{G} (g(x), g(y)) = \mathcal{G} (f^+ (x), f^+ (y))$ for every $x,y\in B_1^+$, we conclude the desired esimate on the H\"older continuity of $f^+$. The one for $f^-$ is slightly more subtle. Consider indeed two points $x,y \in B_1^-$. It then turns out that there are $i, j\in \{1, Q-1\}$ and an invertible map $\sigma : \{1, \ldots , Q-1\}\setminus \{j\} \to \{1, \ldots, Q-1\}\setminus \{i\}$ with the property that
\[
\mathcal{G} (g(x), g(y))^2 = |\phi (x) - f^-_i (y)|^2 + |f^-_j (x) - \phi (y)|^2 + \sum_{k\in \{1, \ldots , Q-1\}\setminus \{j\}} |f^-_k (x) - f^-_{\sigma (k)} (y)|^2\, .
\]
Observe therefore that, by the triangle inequality
\[
|f^-_j (x) - f^-_i (y)|\leq |\phi (x) - \phi (y)| + 2 \mathcal{G} (g (x), g (y))\, .
\]
In particular, using the observation
\[
\mathcal{G} (f^- (x), f^- (y))^2 \leq |f^-_j (x) - f^-_i (y)|^2 + \sum_{k\in \{1, \ldots , Q-1\}\setminus \{j\}} |f^-_k (x) - f^-_{\sigma (k)} (y)|^2\, ,
\]
we achieve
\[
\mathcal{G} (f^- (x), f^- (y))^2 \leq 2|\phi (x) - \phi (y)|^2 + 5 \mathcal{G} (g (x), g (y))^2 \leq 2\|D\phi\|_{C^0}^2 |x-y|^2 + 5 \mathcal{G} (g (x), g (y))^2\, .
\]
Combinining the latter inequality with the estimate for $[g]_{C^{0, \beta}}$ we conclude the desired estimate for the H\"older seminorm of $f^-$. 
\end{proof}

\subsection{Almgren's retractions and maximum principle} An important tool in proving the decay of the Dirichlet energy for $Q$-valued minimizers is a family of retraction maps which can be used, for instance, to prove suitable generalizations of the classical maximum principle for harmonic functions. These maps were introduced by Almgren in his pioneering work and we refer to \cite{DS} for an elementary account of them. In order to deal with $\left( Q-\frac{1}{2} \right)$-maps we need an additional property of such retractions, which is not recorded in \cite{DS} (nor in \cite{Alm}). We start by recalling the following notation:

\begin{defn}[Diameter and separation]
	Let $T= \sum_i \llbracket P_i \rrbracket \in \AQ$. The \textit{diameter} and \textit{separation} of $T$ are defined, respectively, as
	\[ d(T):= \max_{i,j} |P_i - P_j| \text{ and } s(T):= \min \{|P_i - P_j|: P_i \neq P_j \}, \]
	with the convention that $s(T) = +\infty$ if $T= Q\llbracket P \rrbracket$.
	\end{defn}
	
	For $Y = \sum_i \a{P_i}$ we denote by $\supp\, (T)$ the set of points $\{P_1, \ldots , P_Q\}\subset \RR^n$. Clearly
	\begin{equation}\label{eq:defmin}
		\dist (\supp (T), q) = \min_i |P_i -q|. 
	\end{equation} 
	We have a triangle inequality
	\begin{equation}\label{eq:trgmin}
		\dist (\supp (T), q) \leq \dist (\supp (S), q) + \GG(T,S), \quad \text{ for every } T, S \in \AQ. 
	\end{equation} 
	
	\begin{lemma}\label{lm:retraction}
		Let $T\in \AQ$ and $r< s(T)/4$. Then there exists a retraction $\vartheta: \AQ \to \overline{B_r(T)}$ such that
		\begin{enumerate}[(i)]
			\item $\GG(\vartheta(S_1), \vartheta(S_2)) < \GG(S_1, S_2)$ if $S_1 \notin \overline{B_r(T)}$,
			\item $\vartheta(S) = S$ for every $S\in \overline{B_r(T)}$,
			\item If a point $q$ belongs to $\supp (T)$ and to $\supp (S)$, then it belongs to $\supp (\vartheta(S))$ too.
		\end{enumerate}
	\end{lemma}
	\begin{proof}
		We define $\vartheta$ in the same way as \cite[Lemma 3.7]{DS}. The properties (i) and (ii) are proved in \cite[Lemma 3.7]{DS} whereas (iii) is an obvious consequence of the explicit formula given in there.
	\end{proof}
	
	\begin{prop}[Maximum principle]
		Let $f\in W^{1,2}(\Omega, \AQt)$ be a Dir-minimizer with interface $(\gamma, 0)$. Suppose $T\in \AQ$, $0\in \supp\, (T)$ and $0<r< s(T)/4$. If
\begin{align}
\GG(f(x), T) \leq r & \qquad\qquad\mbox{for $\mathcal{H}^{m-1}$-a.e. $x\in (\partial\Omega)^+$ and}\label{e:assum1}\\
\GG (f(x) + \a{0}, T)\leq r & \qquad\qquad \mbox{for $\mathcal{H}^{m-1}$-a.e. $x\in (\partial \Omega)^-$,}\label{e:assum2}\, ,
\end{align}
then
\begin{align}
\GG(f, T) \leq r &\qquad\qquad \mbox{a.e. in $\Omega^+$ and}\\
\GG (f + \a{0}, T) \leq r &\qquad\qquad\mbox{a.e. in $\Omega^-$}\, .
\end{align}
\end{prop}
	
	\begin{proof}
		We argue by contradiction. Suppose $f\in W^{1,2}(\Omega, \AQt)$ is a Dir-minimizer with interface $(\gamma, 0)$ satisfying \eqref{e:assum1} and \eqref{e:assum2} and assume in addition that there exists a set of positive measure $E\subset \Omega$, such that $f(x) \notin \overline{B_r(T)}$ for every $x\in E\cap \Omega^+$ and $f(x) + \{0\}\notin \overline{B_r (T)}$ for every $x\in E \cap \Omega^-$.
	
In particular there exist $\delta>0$ and a set $E' \subset E$ with positive measure such that $f(x) \notin \overline{B_{r+\delta}(T)}$ for every $x\in E'\cap \Omega^+$ and $f(x)+\a{0}\notin \overline{B_{r+\delta} (T)}$ for every $x\in E'\cap \Omega^-$. As in the proof of Lemma \ref{lm:CM} we consider the $Q$-valued function on $\Omega$ which coincides with $f^+$ on $\Omega^+$ and with $f+\a{0}$ in $\Omega^-$. Let $\vartheta: \AQ \to \overline{B_r(T)} $ be the retraction operator in Lemma \ref{lm:retraction}. By (iii) $\supp (\vartheta\circ g(x))$ contains the origin for every $x\in \Omega^-$. We can thus consider the $(Q-1)$-valued function on $\Omega^-$ given by $\vartheta \circ g - \a{0}$. If we set $h^+ = \vartheta \circ g$ on $\Omega^+$ we then get a $(Q-\frac{1}{2})$-valued map $(h^+, h^-)$ with interface $(\gamma, 0)$. By Lemma \ref{lm:retraction}(ii) we also know that $h^\pm = f^\pm$ on $(\partial \Omega)^\pm$. Therefore $h= (h^+, h^-)$ is a suitable competitor for $f=(f^, f^-)$. On the other hand, by Lemma \ref{lm:retraction} (i) we know $|D(\vartheta \circ f)| \leq |Df|$ a.e. on $\Omega$ and moreover, recalling the definition of $\vartheta$ by linear interpolation and that $\GG(f(x), T) > r+\delta$, we get that 
		\begin{equation}\label{eq:strineq}
			|D (\vartheta \circ f)| \leq t_0 |Df| < |Df| \quad \text{ a.e. on } E', 
		\end{equation} 
		where $t_0 \leq \frac{r-\delta}{r+\delta} < 1$. Here we compute the partial derivatives by the first order approximation, see the definition and discussions in Definition 1.9, Corollary 2.7 and Proposition 2.17 of \cite{DS}.
		We conclude that $\Dir(h, \Omega) < \Dir(f, \Omega)$, contradicting the minimality of $f$.
	\end{proof}
	
\subsection{Decomposition} The maximum principle of the previous section triggers a decomposition lemma for Dir-minimizers with $(\gamma, 0)$ interface.

	\begin{prop}[Decomposition of $(Q-\frac{1}{2})$-valued Dir-minimizers]\label{prop:dcp}
		There exists a positive constant $\alpha(Q)>0$ with the following property. Assume that $f\in W^{1,2}(\Omega, \AQt)$ is a Dir-minimizer with interface $(\gamma, 0)$, and that there exists $T\in \AQ$ with $0\in \supp\, (T)$ such that  \eqref{e:assum1} and \eqref{e:assum2} hold with $r= \alpha (Q) f (T)$.
		Then there exists a decomposition $f = (f^+, f^-) = (g^++h, g^- +h)$, where $h$ is a $Q_1$-valued Dir-minimizer, $(g^+, g^-)$ a $(Q_2-\frac{1}{2})$ Dir-minimizer with interface $(\gamma, 0)$, $Q_1+Q_2=Q$ and $1\leq Q_1 \leq Q-1$.
	\end{prop}
	
	\begin{proof}
		When $d(T) = 0$, our assumption implies $\GG(f(x), T) = 0$, namely $f \equiv T$, and there is nothing to prove. So we assume $d(T)>0$.
		If $\alpha(Q) d(T) < s(T)/4$ (for a fixed value of $\alpha(Q)$), the proposition follows directly by the maximum principle and the definition of $s(T)$. Suppose therefore $4 \alpha (Q) d(T) \geq s(T)$. We fix a positive real number $\epsilon$ so that
		\[ \left( \sqrt{Q}+2 \right) \frac{\epsilon}{1-\epsilon} = \frac{1}{8}. \]
		Recalling \cite[Lemma 3.8]{DS}, we may \textit{collapse} some points in the support $T$ and find an element $S= \sum_{j=1}^J k_j \llbracket S_j \rrbracket \in \AQ$ (with $J\geq 2$) satisfying
		\begin{equation}
			\beta(\epsilon, Q) d(T) \leq s(S) < +\infty,
		\end{equation}
		\begin{equation}\label{tmp:dcpQ2}
			\GG(S,T) \leq \epsilon s(S).
		\end{equation}
		We set $\alpha(Q) = \epsilon \beta(\epsilon, Q)$, so that
		\begin{equation}\label{tmp:dcpQ1}
			\GG(f(x), T) \leq \alpha(Q) d(T)\leq \epsilon s(S) \quad \text{ for } \mathcal{H}^{m-1} \text{-a.e. } x\in \partial\Omega. 
		\end{equation} 
		Since $0\in \supp (T)$, we have, by the triangle inequality \eqref{eq:trgmin},
		\[ \dist (\supp\, (S), 0) \leq \min \dist (\supp\, (T), 0) + \GG(S, T) \leq \epsilon s(S). \]
		Without loss of generality, we assume $|S_1| = \dist (\supp (S), 0)$. Let $\widetilde S = k_1 \llbracket 0 \rrbracket + \sum_{j=2}^J k_j \llbracket S_j \rrbracket$. Clearly
		\begin{equation}\label{tmp:dcpQ3}
			\GG(S, \widetilde S) = \sqrt{k_1|S_1|^2} \leq \sqrt{Q} \min S \leq \epsilon \sqrt{Q} \, s(S). 
		\end{equation} 
		On the other hand $s(\widetilde S) \geq (1-\epsilon) s(S)$. In fact, either $s(\widetilde S) = |S_i - S_j|$ for some $i, j \neq 1$, in which case $s(\widetilde S) \geq s(S)$; or $s(\widetilde S) = |S_i|$ for some $i \neq 1$, and then
		\begin{equation}\label{tmp:dcpQ4}
			s(\widetilde S) =  |S_i| \geq |S_i - S_1| - |S_1| \geq s(S) - \epsilon s(S). 
		\end{equation} 
		Combining \eqref{tmp:dcpQ1}, \eqref{tmp:dcpQ2}, \eqref{tmp:dcpQ3}, \eqref{tmp:dcpQ4} and the choice of $\epsilon$, we conclude
		\begin{align*}
			\GG(f(x), \widetilde S) \leq \GG(f(x), T) + \GG(S, T) + \GG(S, \widetilde S) & \leq \epsilon s(S) + \epsilon s(S) + \epsilon \sqrt{Q} \, s(S) \\
			& \leq \left( \sqrt{Q} + 2 \right) \frac{\epsilon}{1-\epsilon} \, s(\widetilde S)
			 = \frac18 s(\widetilde S),
		\end{align*}
		for $\mathcal{H}^{m-1}$-a.e. $x\in \partial\Omega$. Again it follows by the maximum principle that $\GG(f, \widetilde S) \leq s(\widetilde S)/8$ almost everywhere on $\Omega$. We thus have a decomposition of $f$ into simpler multiple-valued functions.
	\end{proof}

\subsection{Interpolation that preserves the interface value}

In this subsection, we construct interpolations between pairs of $(Q-\frac{1}{2})$ maps with a common interface $(\gamma, 0)$ defined on concentric spheres and estimate its Dirichlet energy. Later we will use the interpolation to construct competitors for Dir-minimizing maps, so it is crucial that the interpolation has the same interface $(\gamma, 0)$. This is also the major difference from the interior case, proved in \cite[Lemma 2.15]{DS}. For our current purpose, namely the proof of the decay of the Dirichlet energy for minimizers, we actually need the existence of the interpolation only in the case $m\geq 3$. However later on Lemma \ref{lm:interpolation} will be used on planar maps to show the compactness of minimizers, a crucial point in the proof of Theorem \ref{thm:main}. We therefore state and proof also the $2$-dimensional case (separately). 

\begin{lemma}[Interpolation when $m=2$]\label{lm:interpolation}
	Let $f, g$ be maps in $ W^{1,2}(\partial B_1, \AQt(\RR^n))$ satisfying 
	\begin{equation}\label{eq:tmpbd2}
		f^+|_{\gamma} = f^-|_{\gamma} + \llbracket 0 \rrbracket, \quad g^+|_{\gamma} = g^-|_{\gamma} + \llbracket 0 \rrbracket, 
	\end{equation} 
	and $\sup_{x\in \partial B_1} \GG(f(x), g(x)) < +\infty $. Let $\delta= \frac{1}{N}$ for some $N\in \mathbb N\setminus \{0,1,2,3\}$. Then there exists $h\in W^{1,2}(B_1\setminus B_{1-\delta}, \AQt(\RR^n))$ satisfying $h^+|_{\gamma} = h^-|_{\gamma} + \llbracket 0 \rrbracket$ and
	\[ h(x) = f(x) \text{ for } x\in \partial B_1, \quad h(x) = g \left( \frac{1}{1-\delta}x \right) \text{ for } x\in \partial B_{1-\delta}. \]
	Moreover
	\begin{equation}\label{eq:interpolation2}
		\Dir(h, B_1 \setminus B_{1-\delta}) \leq C\delta \Dir(f, \partial B_1) + C\delta \Dir(g, \partial B_1) + \frac{C}{\delta} \sup_{x\in \partial B_1} \GG(f(x), g(x)).
	\end{equation}
\end{lemma}

\begin{proof} By applying a diffeomorphism, we can assume that $\gamma = \mathbb R$. 
	We first interpolate $f^+$ and $g^+$ in the upper half annulus $B_1^+\setminus B_{1-\delta}^+$. After parametrizing a biLipschitz diffeomorphism $\phi: [0,1] \to \partial B_1^+$ to the functions $f^+$ and $g^+$, we may assume $f^+, g^+$ are $W^{1,2}$ maps defined on $[0,1]$. We will interpolate $f^+$ and $g^+$ and get a $W^{1,2}$ map on $[0,1]\times [0,\delta]$. 
	
	We define a cubical decomposition $D_i = [i\delta, (i+1)\delta]\times [0,\delta]$ with $i=0,1, \cdots, N-1$, and vertical lines $\ell_i = \{ i\delta\} \times [0,\delta]$ with $i=0, 1, \cdots, N$. For $i=1,2, \cdots, N-1$, we define
	\[ h(x,t)= \bxi^{-1} \circ \brho \left( \left(1-\frac{t}{\delta} \right) \bxi \circ g^+(x) + \frac{t}{\delta} \bxi \circ f^+(x) \right), \quad (x,t) \in \ell_i,  \]
	where $\bxi: \AQ(\RR^n) \to \RR^N$ is the embedding of $Q$-valued metric space, and $\brho: \RR^N \to \bxi( \AQ)$ is the retraction, see \cite[Theorem 2.1]{DS}. It is clear that
	\begin{equation}
		\left| Dh(x,t) \right| \leq \frac{C}{\delta} \GG(g^+(x), f^+(x)),
	\end{equation}
	where the constant depends on the Lipschitz constants of $\bxi$ and $\brho$. 
	For $i=0$ or $N$ and $x=i\delta$, by recalling \eqref{eq:tmpbd2} we denote 
	\begin{equation}\label{eq:tmpbdexplicit}
		g^+(x) = \sum_{j=1}^Q \llbracket a_j\rrbracket = \llbracket 0 \rrbracket + g^-(x), \quad f^+(x) = \sum_{j=1}^Q \llbracket b_j\rrbracket = \llbracket 0 \rrbracket + f^-(x). 
	\end{equation} 
	Here we assume $a_1 = b_1 = 0$ without loss of generality.
	Suppose $\tau$ is a permutation of $\{2, \cdots, Q\}$ such that
	\[ \GG(g^-(x), f^-(x)) = \sqrt{\sum_{j=2}^Q |a_j - b_{\tau(j)}|^2 }. \]
	We define
	\begin{equation}\label{eq:tmp}
		h(x,t) = \llbracket 0 \rrbracket + \sum_{j=2}^Q \left \llbracket \left( 1- \frac{t}{\delta} \right) a_j + \frac{t}{\delta} b_{\tau(j)}  \right\rrbracket.
	\end{equation}
	\eqref{eq:tmpbdexplicit} implies that
	\begin{equation}\label{eq:chgaftzero}
		\GG(g^+(x), f^+(x)) \leq \GG(g^-(x), f^-(x)) \leq \sqrt{2} \GG(g^+(x), f^+(x)). 
	\end{equation} 
	Hence
	\begin{equation}
		\left|Dh(x,t) \right| = \frac{1}{\delta} \sqrt{\sum_{j=2}^Q |a_j - b_{\tau(j)}|^2 } = \frac{1}{\delta} \GG(g^-(x), f^-(x)) \leq \frac{\sqrt{2}}{\delta} \GG(g^+(x), f^+(x)).
	\end{equation}
	
	In this way $h$ is well-defined for each $\partial D_i$. We now wish to use \eqref{eq:Dirbd} (and a biLipschitz homeomorphism of squares to disks) and claim the existence of an extension $h$ on $D_i$ satisfying
	\begin{equation}
		\Dir(h, D_i) \leq C \delta \Dir(h, \partial D_i).
	\end{equation}
Note that this can be done because the proof of \eqref{eq:Dirbd} given later in the planar case is {\em not} using the current proposition (it uses interpolation, however, if the domain is at least $3$-dimensional).
	Summing up we get
	\begin{align*}
		\Dir(h, [0,1]\times [0,\delta]) & = \sum_{i=0}^{N-1} \Dir(h, D_i) \\
		 & \leq C \delta \left(  \Dir(h, [0,1]\times \{0\}) + \Dir(h,[0,1] \times \{\delta\}) + \sum_{i=0}^N \Dir(h, \ell_i) \right) \\
		 & \leq C\delta \Dir(g,[0,1]) + C\delta\Dir(f,[0,1]) + C \sum_{i=0}^N \GG(g^+( i \delta), f^+(i\delta)) \\
		 & \leq C\delta \Dir(g,[0,1]) + C\delta\Dir(f,[0,1]) + \frac{C}{\delta} \sup_{x\in [0,1]} \GG(g^+( x), f^+(x)).
	\end{align*}
	Applying the biLipschitz homeomorphism $\phi:[0,1] \to \partial B_1^+$, we get an interpolation $h^+ \in W^{1,2}(B_1^+ \setminus B_{1-\delta}^+, \AQ)$.
	
	Similarly, we define an interpolation $h^- \in W^{1,2}(B_1^-\setminus B_{1-\delta}^+, \mathcal{A}_{Q-1})$ between $g^-$ and $f^-$. By \eqref{eq:tmpbd} and the construction \eqref{eq:tmp}, we know $h^+|_{\gamma} = h^-|_{\gamma} + \llbracket 0 \rrbracket$, $h\in W^{1,2}(B_1 \setminus B_{1-\delta}, \AQt)$ and moreover
	\begin{equation}
		\Dir(h, B_1\setminus B_{1-\delta}) \leq C\delta \Dir(g, B_1) + C\delta\Dir(f, B_1) + \frac{C}{\delta} \sup_{x\in \partial B_1} \GG(g(x), f(x)).
	\end{equation}
\end{proof}

\begin{lemma}[Interpolation when $m\geq 3$]\label{lm:interpolation}
	Let $f, g$ be maps in $ W^{1,2}(\partial B_1, \AQt(\RR^n))$ satisfying 
	\begin{equation}\label{eq:tmpbd}
		f^+|_{\gamma} = f^-|_{\gamma} + \llbracket 0 \rrbracket, \quad g^+|_{\gamma} = g^-|_{\gamma} + \llbracket 0 \rrbracket, 
	\end{equation} 
	and $\int_{ \partial B_1} \GG(f, g) < +\infty $. Let $\delta= \frac{1}{N}$ for some $N\in \mathbb N\setminus \{0,1,2,3\}$. Then there exists $h\in W^{1,2}(B_1\setminus B_{1-\delta}, \AQt(\RR^n))$ satisfying $h^+|_{\gamma} = h^-|_{\gamma} + \llbracket 0 \rrbracket$ and
	\[ h(x) = f(x) \text{ for } x\in \partial B_1, \quad h(x) = g \left( \frac{1}{1-\delta}x \right) \text{ for } x\in \partial B_{1-\delta}. \]
	Moreover
	\begin{equation}\label{eq:interpolation}
		\Dir(h, B_1 \setminus B_{1-\delta}) \leq C\delta \Dir(f, \partial B_1) + C\delta \Dir(g, \partial B_1) + \frac{C}{\delta} \int_{ \partial B_1} \GG(f, g).
	\end{equation}
\end{lemma}

\begin{proof} By applying a diffeomorphism, we can assume that $\gamma = \{x_m =0\}$. 
%
	Let $\mathcal{C}$ be the boundary of the cube $[-1, 1]^m$. Notice that $\mathcal{C}$ is tangent to the sphere $\partial B_1$. 
	We define the functions $\hat f$ and $\hat g$ on $\mathcal{C}$ by radial projection:
	\[ \hat f(z) := f\left( \frac{z}{|z|} \right), \quad \hat g(z) := g\left( \frac{z}{|z|} \right), \quad \text{ for every } z\in \mathcal{C}. \]
	After the radial projection, the tangential derivative on $\mathcal{C}$ at $z$ is just a multiple of the tangential derivative on $\partial B_1$ at $z/|z|$, where the factor is uniformly bounded above and below by dimensional constants. In particular $\hat f, \hat g \in W^{1,2}(\mathcal{C}, \AQt)$, that is, 
	\[ \hat f^+, \hat g^+ \in W^{1,2}(\mathcal{C}^+, \AQ), \quad \hat f^-, \hat g^- \in W^{1,2}(\mathcal{C}^-, \mathcal{A}_{Q-1}), \]
	and
	\[ \hat f^+|_\gamma = \hat f^-|_\gamma +\llbracket 0 \rrbracket, \quad \hat g^+|_\gamma = \hat g^-|_\gamma +\llbracket 0 \rrbracket,  \]
	where $\mathcal{C}^+ = \mathcal{C}\cap \{x_m >0\}$, $\mathcal{C}^- = \mathcal{C} \cap \{x_m <0\}$ and $\gamma = \{x_m = 0\}$.
%
%
	 We want to construct a function $\hat h: \mathcal{C} \times [0,\delta] \to \AQt$ which satisfies $\hat h(\cdot, 0) = \hat g$, $\hat h(\cdot, \delta) = \hat f$, $\hat h \in W^{1,2}$ and
	 \[ \hat h^+|_{\gamma \times[0,\delta]} = \hat h^-|_{\gamma \times[0,\delta]} + \llbracket 0 \rrbracket; \] and in turn, we define a function $h: B_1\setminus B_{1-\delta} \to \AQt $ by
	 \[ h\left( t\, \frac{z}{|z|} \right) := \hat h\left(z, t-(1-\delta) \right), \quad \text{ for each } z\in \mathcal{C} \text{ and } 1-\delta < t < 1, \]
	 such that $h\in W^{1,2}(B_1 \setminus B_{1-\delta}, \AQt)$ with the desired boundary data.
	 
	 Let $F$ be any of the $2m$ faces of $\mathcal{C}$, then it is an $(m-1)$-dimensional solid cube (i.e. including the interior) with side length $2$. Take for example 
	 \[ F=\left\{\left(-1, x_2, \cdots, x_m \right): -1 \leq x_j \leq 1 \text{ for every } j=2, \cdots, m \right\}. \] 
	 We will first define $\hat h$ on $F\times [0, \delta]$ using the similar construction as in the interior case, see Step 1 of \cite[Lemma 4.12]{DS} and the erratum therein. To that end we first need to extend $\hat f$ and $\hat g$ to a fatter region
	 \[ F_\delta:= \left\{\left(-1, x_2, \cdots, x_m \right): -1-\delta \leq x_j \leq 1 + \delta \text{ for every } j=2, \cdots, m \right\}, \]
	  by using their respective values on neighboring faces of $F$ and scaling appropriately on the corners. For example, for any $x_2 \in [-1 - \delta, - 1)$ fixed (the other possibility being $x_2 \in (1, 1+\delta]$), we consider the slice 
	  \[ S_{x_2} := \left\{ \left(-1, x_2, x_3, \cdots, x_m \right): -|x_2| \leq x_j \leq |x_2| \text{ for every } j=3, \cdots, m \right\} \subset F_{\delta}, \]
	  and define $\hat f, \hat g$ by their values on a neighboring face of $F$:
	  \begin{equation}\label{eq:defFng}
	  	F':= \left\{(x_1, -1, x_3, \cdots, x_m): -1\leq x_j \leq 1 \text{ for every } j= 1, 3, \cdots, m \right\}. 
	  \end{equation} 
	  To be precise on $S_{x_2}$ we define
	  \begin{equation}\label{eq:fatdef1}
	  	\hat f\left(-1, x_2, x_3, \cdots, x_m \right) := \hat f \left( |x_2|-2,\, - 1, \, \varphi_\delta( x_3), \cdots, \varphi_\delta(x_m) \right) 
	  \end{equation} 
	  where $\varphi_\delta: [-|x_2|, |x_2|] \to [-1, 1]$ is a piecewise linear function as follows
	  \begin{equation}\label{eq:fatdef2}
	  	\varphi_\delta(t) = \left\{\begin{array}{ll}
	  	-1 + \dfrac{\delta}{-1 + \delta+ |x_2|}\left( t + |x_2| \right), & -|x_2| \leq t \leq -1 + \delta \\
	  	t, & -1 + \delta \leq t \leq 1 -\delta \\
	  	1 + \dfrac{\delta}{-1 + \delta + |x_2|} \left( t-|x_2| \right), & 1 - \delta \leq t \leq |x_2|.
	 	  \end{array}\right. 
	  \end{equation} 
	That is, in the inner region of $S_{x_2}$, $\hat f$ (as well as $\hat g$) takes value on $F'$ faithfully; in the outer region $\hat f$ (as well as $\hat g$) is a scaled version of its value on $F'$, with a scaling factor at most $2$. The former is to guarantee that the construction of $\hat h$ remains faithful to $\hat f, \hat g$ near the boundary $\gamma \times [0,\delta]$.

	
	For any vector $v\in [-1-\delta,-1]^{m-1}$, consider the cubical decomposition of $F_\delta$ induced by the lattice points $\{- 1 \} \times \left( v+ \delta \mathbb{Z}^{m-1} \right)$. For $k\in \{0, \cdots, m-1\}$ we define accordingly the $k$-dimensional skeleta contained in $F_\delta$, which are the families $\mathcal{S}^k(v)$ of all closed $k$-dimensional faces of the cubes. By Fubini, for almost every $v$ and face $E\in \mathcal{S}^k(v)$, we have that $\hat f|_E, \hat g|_E \in W^{1,2}$, and moreover
	\begin{align*}
& \int_{v\in [-1 - \delta, -1]^{m-1} } \left( \sum_{E\in \mathcal{S}^k(v)} \int_E \left( |D\hat f|^2 + |D\hat g|^2 + \mathcal{G}(\hat f,\hat g)^2 \right) \right) dv\\
 \leq & C(k,m) \delta^{k } \int_{F_{\delta}} \left( |D\hat f|^2 + |D\hat g|^2 + \mathcal{G}(\hat f,\hat g)^2 \right). 
\end{align*}
	By standard arguments we can choose a vector $v$ such that
	\begin{itemize}
		\item For every $k\geq 1$, for each $E\in \mathcal{S}^k(v)$ and each $G\in \mathcal{S}^{k-1}(v)$ with $G\subset E$, the restrictions $\hat f|_E, \hat f|_G, \hat g|_E, \hat g|_G$ are all $W^{1,2}$ and moreover the traces of $\hat f|_E$ and $\hat g|_E$ on $G$ are precisely $\hat f|_G$ and $\hat g|_G$;
		\item For every $k\geq 1$,
			\[ \sum_{E\in \mathcal{S}^k(v)} \int_E \left( |D\hat f|^2 + |D\hat g|^2 \right) \leq C \delta^{k-(m-1)} \int_{F_{\delta}} \left( |D\hat f|^2 + |D\hat g|^2 \right); \]
		\item For $k=0$,
			\[ \sum_{p\in \mathcal{S}^0(v)} \mathcal{G}(\hat f(p), \hat g(p))^2 \leq C\delta^{-(m-1)} \int_{F_\delta} \mathcal{G}(\hat f,\hat g)^2; \]
		\item Whenever $E\in \mathcal{S}^{k}(v)$ intersects $\gamma$, the center of $E$, denoted by $x_E$, lies in $\mathcal{C}^+$, in other words $x_E$ lies above the boundary $\gamma$.
	\end{itemize}
	For any $k=0, \cdots, m-1$ and any $E\in \mathcal{S}^{k}(v)$ not intersecting $\gamma$, we follow the same construction as in the interior case (for $Q$-valued or $(Q-1)$-valued functions) and define $\hat h$ on $E \times [0,\delta]$ by interpolation of $\hat f^+$ and $\hat g^+$, or $\hat f^-$ and $\hat g^-$ respectively. Across the boundary $\gamma$, we temporarily extend the functions trivially by zero, that is, we set
	\[ \hat f_0 = \left\{\begin{array}{ll}
		\hat f^+, & \text{ on } \mathcal{C}^+ \\
		\hat f^- + \llbracket 0 \rrbracket, & \text{ on } \mathcal{C}^-,
	\end{array} \right.
	\quad
	\hat g_0 = \left\{\begin{array}{ll}
		\hat g^+, & \text{ on } \mathcal{C}^+ \\
		\hat g^- + \llbracket 0 \rrbracket, & \text{ on } \mathcal{C}^-,
	\end{array} \right. \]
	 so that $\hat f_0, \hat g_0$ are $Q$-valued functions. Notice that the values of $|D\hat f|,\, |D\hat g|,\, \mathcal{G}(\hat f^+, \hat g^+)$ stay the same, and on $\mathcal{C}^-$
	 \[ \frac{1}{\sqrt{2}} \mathcal{G}(\hat f^-, \hat g^-) \leq \mathcal{G}(\hat f_0, \hat g_0) \leq \mathcal{G}(\hat f^-, \hat g^-), \]
	  see \eqref{eq:chgaftzero}.
	  Recall that for any $p\in \mathcal{S}^0(v)$ contained in $\mathcal{C}^-$, we define $\hat h$ on $p\times [0,\delta]$ as a linear interpolation between $\hat f$ and $\hat g$, and that
	  \begin{equation}\label{eq:intpbase}
	  	\Dir(\hat h, p\times[0,\delta]) \leq \frac{C}{\delta} \mathcal{G}(\hat f^-(p), \hat g^-(p))^2 \leq \frac{C'}{\delta} \mathcal{G}(\hat f_0(p), \hat g_0(p))^2. 
	  \end{equation} 
	  
	  Now we construct $\hat h$ by an induction on the dimension $k$.  Suppose $E\in \mathcal{S}^k (v)$ intersects $\gamma$, where $k = 1, \cdots, m-1$. Either by the inductive hypothesis or by the base case $k=0$ (see \eqref{eq:intpbase} and assume $\hat h_0(p) = \hat h(p) + \llbracket 0 \rrbracket$), we assume that for all lower skeleta $G\in \mathcal{S}^{k-1}(v)$ with $G\subset E$, we have defined a $Q$-valued function $\hat h_0$ on $G \times [0,\delta]$ with the desired properties. Since 
	  \[ \partial\left( E\times [0,\delta] \right) = \bigcup_{G\in \mathcal{S}^{k-1}(v) \atop{G\subset E}} \left( G\times[0,\delta]\right) \, \bigcup \, \left(E\times \{0\} \right) \, \bigcup \, \left( E \times \{\delta\} \right), \] 
	  we can define $\hat h_0$ on $E\times [0,\delta]$ as the $0$-homogeneous extension of $\hat h_0|_{\partial \left( E\times [0,\delta]\right)}$. Simple computations show that
	  \[ \Dir(\hat h_0, E \times[0,\delta] ) \leq C \delta \Dir\left(\hat h_0, \partial\left( E \times [0,\delta] \right) \right). \]
	  More importantly, notice that every point on $\left( E \cap \mathcal{C}^-\right) \times (0,\delta)$ lies in a line segment between the center $x_E \times \{\delta/2\}$ and some point in $\left( E \cap \mathcal{C}^- \right) \times \{0\}$, $\left( E \cap \mathcal{C}^- \right) \times \{\delta\}$ or $\left(G\cap \mathcal{C}^- \right) \times [0,\delta]$ for some $G \in \mathcal{S}^{k-1}(v)$ and $G \subset E$, hence this construction guarantees that on $\mathcal{C}^- \times [0,\delta]$, the $Q$-valued function $\hat h_0$ always has an element $\llbracket 0 \rrbracket$; in particular we may define $\hat h \in \AQt$ accordingly and it satisfies the desired boundary condition. To sum up, we construct a function $\hat h_F$ defined on $\widetilde F_\delta \times [0,\delta]$, where $F \subset \widetilde F_\delta \subset F_\delta$, and it satisfies 
	  \[ \hat h_F(\cdot, 0) = \hat g, \quad \hat h_F(\cdot, \delta) = \hat f \text{ on } F; \]
	  \[ \hat h_F^+(\cdot, t)\big|_{\gamma} = \hat h_F^-(\cdot, t)\big|_{\gamma} + \llbracket 0 \rrbracket \quad \text{ for every } t\in [0,\delta]; \]
	  \[ \Dir(\hat h_F, F\times [0,\delta]) \leq C \delta \Dir(\hat f, F_\delta) + C\delta \Dir(\hat g, F_\delta) + \frac{C}{\delta} \int_{F_\delta } \mathcal{G}(\hat f, \hat g)^2. \]
	  
	  We would like to repeat the same argument for any neighboring face of $F$, take for example $F'$ as in \eqref{eq:defFng}; but we need to be careful and make sure the new function $\hat h_{F'}$ is consistent with $\hat h_F$ on their domains of overlap, since $\hat h_F$ is defined on a small neighborhood near $F \cap F'$ by projecting the fattened region $\widetilde F_\delta$ onto $F'$:
	  \[ \Ng(F):= F' \cap \left\{-1 \leq x_1 \leq -1 + \delta' \right\},\]
	  where $\delta' \in [0, \delta)$ is determined by the choice of $v$.
	  
	   We sketch the necessary technical modifications below. As before, we consider a fattened region $F'_\delta$ of $F'$; and we then choose a cubical decomposition of $F'_\delta$ to satisfy, in addition to the requirements stated above, that all skeleta (orthogonal to $x_1$-axis) ought to be at least $\delta/2$-distance away from $\Ng(F)$. On the interior region
	   \[ \Ng^i(F):= \Ng(F) \cap \left\{ - 1 + \delta \leq x_j \leq 1 - \delta \text{ for every } j= 3, \cdots, m \right\}, \]
	   we use $\hat h_F$ as the boundary condition to construct $\hat h_{F'}$ to make sure they agree; outside, on each $(m-1)$-dimensional $\delta$-cube $E$ contained in $\Ng(F) \setminus \Ng^i(F)$, we replace and reconstruct $\hat h_F$ on $E\times [0,\delta]$ as above. This way $\hat h_F = \hat h_{F'}$ on their domains of overlap $\Ng(F)$; moreover, since we do not redefine $\hat h_F$ near the boundary $\gamma \times [0,\delta]$, it still satisfies the desired boundary condition.
%
\end{proof}

\subsection{Decay estimate} The key point in the proof of Theorem \ref{thm:Holder} is a suitable decay estimate for the Dirichlet energy, which is essentially the content of the following proposition.

\begin{prop}\label{prop:Dirbd}
	Suppose $f$ is a $(Q-\frac{1}{2})$ Dir-minimizing map on $B_1$ with interface $(\gamma, \varphi)$ and assume that $\gamma$ is the graph of a function $\zeta$ with $\|\zeta\|_{C^1}\leq 1$. Let $0<r<1$ and assume that $f|_{\partial B_r} \in W^{1,2}(\partial B_r, \AQt)$.
	Then we have
	\begin{equation}\label{eq:Dirbd}
		\Dir(f,B_r) \leq C(m) r \Dir (f, \partial B_r) + C r^m \|D\varphi\|_{C^0}^2\, ,
	\end{equation}
	where $C(m) < (m-2)^{-1}$.
\end{prop}
\begin{remark}\label{rmk:Dirbd}
	By translation, the same estimate holds for any ball $\overline{B_r(y)} \subset B_1$ with $y\in \gamma$. If $\overline{B_r(y)} \cap \gamma = \emptyset$, the analogous interior estimate was proven in \cite[Proposition 3.10]{DS}. 

\end{remark}
\begin{proof}
	 We will prove \eqref{eq:Dirbd} for $r=1$, because the general case follows from a scaling argument. Moreover we will assume, without loss of generality, that $\varphi \equiv 0$. Indeed, for a general $\varphi$, we let $\phi$ be an extension to $B_1$ with the property that $\|D\phi\|_{C^0(B_1)} \leq C \|D\varphi\|_{C^0(\gamma)}$, since the interface $\gamma$ is given by the graph of $\zeta$ satisfying $\|\zeta\|_{C^1} \leq 1$. Define then $(h^+, h^-)$ as 
	 \[
	 h^\pm (x) = \sum_i \a{f^\pm (x) - \phi (x)}\, .
	 \]
	 Moreover, let $k^\pm$ be a Dir-minimizer with boundary values $h^\pm$ and interface $(\gamma, 0)$ and construct a corresponding competitor for $f$ by setting
	 \[
	 \bar{g}^\pm (x) = \sum_i \a{k^\pm +\phi (x)}\, .
	 \]
	 Observe that for every $\varepsilon$ there is a constant $C(\varepsilon)$ such that
	 \begin{align*}
	 |D_\tau h^\pm (x)|^2 &\leq (1+\varepsilon) |D_\tau f^\pm (x)|^2 + C (\varepsilon) |D_\tau \phi (x)|^2, \\
	 |D \bar g^\pm (x)|^2 &\leq (1+\varepsilon) |D k^\pm (x)|^2  + C (\varepsilon) |D \phi (x)|^2.
	 \end{align*}
	 Here $D_\tau$ denotes the tangential derivative on the boundary $\partial B_1$.
	 After proving the Proposition for interfaces $(\gamma, 0)$ we will know that there is a constant $C' (m) < \frac{1}{m-2}$ such that
	 \begin{align*}
	 \Dir (k, B_1) &\leq C'(m) \Dir (k, \partial B_1) = C'(m) \Dir (h, \partial B_1) \leq C' (m)( 1 + \varepsilon) \Dir (f, \partial B_1) + C (m, \varepsilon) \|D\phi\|_{C^0}^2.   
	 \end{align*}
	 Hence we could estimate
	 \begin{align*}
	 \Dir (f, B_1) &\leq \Dir (\bar g, B_1) \leq (1+\varepsilon) \Dir (k, B_1) + C (\varepsilon) \|D\phi\|_{C^0}^2\\
	 &\leq C'(m) (1+\varepsilon)^2 \Dir (f, \partial B_1) + C' (m, \varepsilon) \|D\phi\|_{C^0}^2\, .
	 \end{align*}
	 Since $C'(m) < \frac{1}{m-2}$ it suffices to choose $\varepsilon$ so that $C(m) := C'(m) (1+\varepsilon)^2 < \frac{1}{m-2}$.
	 From now on we restrict therefore our attention to the case $\varphi \equiv 0$.
	 
	 \medskip
	 
	 \textbf{The planar case.} Set $g:= f|_{\partial B_1}$ and let $g= g_0 + \sum_{j=1}^J g_j$ be a decomposition into irreducible maps as in Proposition \ref{prop:decomp}. Suppose $g_0$ unwinds to $\zeta_0: \So \to \RR^n$ as in Proposition \ref{prop:decomp} (ii); and each $g_j$ unwinds to a $W^{1,2}$ function $\zeta_j: \So \to \RR^n$ as in \cite[Proposition 1.5 (ii)]{DS}:
	\[ g_j(x) = \sum_{z^{Q_j} = x } \llbracket \zeta_j(z) \rrbracket. \]
	
	Now we construct an admissible competitor for $f$ as follows. Recall that $\zeta_0(0) = \zeta_0(2\pi) = 0$, we consider its Fourier expansion
	\[ \zeta_0(\theta) = \sum_{l=1}^\infty c_l \sin \left( \frac{l\theta}{2} \right). \]
	We then extend $\zeta_0$ to be a $W^{1,2}$ function defined on all of $B_1$ as:
	\[ \overline\zeta_0(r,\theta) = \sum_{l=1}^\infty r^{\frac{l}{2}} c_l \sin \left( \frac{l\theta}{2} \right). \]
	Note that $\overline\zeta_0$ is not harmonic, but it vanishes on all of the positive real axis. We also consider the harmonic extension of each $\zeta_j$, denoted by $\overline\zeta_j$. Simple computations show that
	\begin{equation}\label{tmp:ext}
		\iint_{\DD} |D\overline\zeta_0|^2 \leq 2 \int_{\So} |\partial_\tau \zeta_0 |^2, \quad \iint_{\DD} |D\overline{\zeta_j}|^2 \leq \frac12 \int_{\So} |\partial_\tau \zeta_j |^2.
	\end{equation}	
	We then unroll $\overline\zeta_0$ to a $(Q_0 - \frac12)$-valued function $h_0 = (h_0^+, h_0^-)$ as in Lemma \ref{lm:unrolling}. By definition, it follows that $h_0$ satisfies the boundary condition
	\[ h_0^+|_{\gamma} = h_0^-|_{\gamma} + \llbracket 0 \rrbracket. \]
	 We also unroll each $\overline\zeta_j$ to $Q_j$-valued function $h_j$ by 
	\[ h_j(x) = \sum_{z^{Q_j} = x} \llbracket \overline\zeta_j (z) \rrbracket. \]
	The function $h = (h_0^+, h_0^-) + \sum_{j=1}^J h_j$ has interface $(\gamma, \varphi)$, agrees with $f$ on $\So$, and thus is an admissible competitor for $f$ in $B_1$. Therefore by Lemma \ref{lm:unrolling}, \cite[Lemma 3.12]{DS} and \eqref{tmp:ext}, we get
	\begin{align*}
		\Dir(f, B_1) \leq \Dir(h, B_1)  &= \sum_{j=0}^J \Dir(h_j, B_1)  = \sum_{j=0}^J \iint_{\DD} |D\overline \zeta_j|^2 
	 \leq 2 \sum_{j=0}^J \int_{\So} |\partial_\tau  \zeta_j|^2 \\
		& = (2Q_0-1) \Dir(g_0, \So) + \sum_{j=1}^J 2Q_j \Dir(g_j, \So) 
		 \leq 2Q \Dir(g, \partial B_1).
	\end{align*}
	In particular, the above inequality says that the constant in \eqref{eq:Dirbd} satisfies $C(2) = 2Q(1+\epsilon)^2$, and we may assume that $C(2) = 3Q$ for example.

	\textbf{The non-planar case.} We define $Q$-valued functions $\tilde{g}$ and $\tilde{f}$ by adding a ``$0$ sheet'' to $g^-$ and $f^-$, as in \eqref{eq:extension}. Observe that $|D\tilde g (x)| = |Dg^\pm (x)|$ and $|D_\tau \tilde{f} (x)| = |D_\tau f (x)|$. So, rather than exhibiting a competitor for $g$ we wish to exhibit a competitor, say $h$, for $\tilde{g}$: we just have to respect the property that $\supp\, h (x)\ni 0$ for every $x\in B_1^-$. With a slight abuse of notation we thus keep the notation $g$ and $f$ for $\tilde{g}$ and $\tilde{f}$.  
	
	\textit{Step 1. Radial competitors.} Let $\bar g= \sum_i \llbracket \bar g_i \rrbracket \in \AQ$ be a mean for $g$ so that the Poincar\'e inequality of \cite[Proposition 2.12]{DS} holds, i.e.
	\begin{equation}\label{eq:Poincarebd}
		\left( \int_{\partial B_1} \GG(g, \bar g)^{p} \right)^{1/p} \leq C (p) \left( \int_{\partial B_1} |Dg|^2 \right)^{1/2},
	\end{equation}
	where the exponent $p$ can be taken to be any finite real $p\geq 1$ if $m=3$ and any real $1\leq p\leq 2^*$ (with $\frac{1}{2^*} = \frac12 - \frac{1}{m-1}$) when $m\geq 4$. Assume the diameter of $\bar g$ is smaller than a constant $M>0$ (whose value is to be determined later),
	\[ d(\bar g) \leq M. \]

	Recall next \eqref{eq:defmin} and define the function $m (x):= \dist (\supp\, g (x), 0)$.  Observe that 
\[
T \;\mapsto\; \dist (\supp\, (T), 0)
\] 
is a Lipschitz map with constant less or equal than $1$ by \eqref{eq:trgmin}: thus $|D m|\leq |Dg|$ and $|D_\tau m|\leq |D_\tau g|$. Moreover $m$ obviously vanishes on $\partial B_1^-$ (whose surface measure is larger than a geometric constant). By the relative Poincar\'e inequality, we know
	\begin{equation}
		\int_{\partial B_1}  m(x)^2 \leq C \int_{\partial B_1} | D m (x)|^2 \leq C \int_{\partial B_1} |Dg(x)|^2.
	\end{equation}
	Hence
	\begin{equation}\label{eq:ming}
		\bar m ^2 := \dist (\supp ( \bar g), 0)^2 = \fint_{\partial B_1} m^2 \lesssim \int_{\partial B_1} m^2 + \int_{\partial B_1} \GG(g(x), \bar g)^2 \leq C \int_{\partial B_1} |Dg(x)|^2.
	\end{equation}
	Combined with the assumption $d(\bar g) \leq M$, it follows that
	\[ |\bar g|^2 = \sum_i |\bar g_i|^2 \leq Q(C+M^2). \]
	Thus
	\begin{align*}
		\int_{\partial B_1} |g|^2 \leq 2 \int_{\partial B_1} \GG(g, \bar g)^2 + 2 \int_{\partial B_1} |\bar g|^2 \leq C_{Q,M},
	\end{align*}
	where $C_{Q,M}$ is a constant depending on $Q$ and $M$ with positive correlation.
	Let $\varphi$ be a real-valued function in $W^{1,2}([0,1])$ with $\varphi(1) = 1$. Then 
	\[ \hat f(x) := \varphi(|x|) g\left( \frac{x}{|x|} \right)\] is a suitable competitor for $f$. A simple computation shows that
	\begin{align*}
		\iint_{B_1} |D \hat f|^2 & = \left( \int_{\partial B_1} |g|^2 \right) \int_0^1 \varphi'(r)^2 r^{m-1} dr + \left( \int_{\partial B_1} |Dg|^2 \right) \int_0^1 \varphi(r)^2 r^{m-3} dr \\
		& \leq \int_0^1 \left( \varphi(r)^2 r^{m-3} + C_{Q,M} \varphi'(r)^2 r^{m-1} \right) dr =: I(\varphi).
	\end{align*}
	By minimality we deduce that
	\[ \Dir(f, B_1) \leq \inf_{\varphi \in W^{1,2}([0,1]) \atop{\varphi(1) = 1}} I(\varphi). \]
	We notice that $I(1) = \frac{1}{m-2}$ ($\varphi \equiv 1$ corresponds to the trivial radial competitor for $f$). On the other hand $\varphi \equiv 1$ can not be a minimum for $I$ because it does not satisfy the corresponding Euler-Lagrange equation. So there exists a constant $\gamma = \gamma(Q, M)>0 $ such that
	\[ \Dir(f, B_1) \leq \inf_{\varphi \in W^{1,2}([0,1]) \atop{\varphi(1) = 1}} I(\varphi) = \frac{1}{m-2} - \gamma. \]
	
	In particular, when $Q=1$, the diameter $d(\bar g) = 0$ and we are done. We will prove the proposition by an induction on $Q$.
	
	\textit{Step 2. Splitting procedure: the inductive step.} Let $Q\geq 2$ be fixed and assume that the proposition holds for every $Q^* < Q$. Assume moreover that $d(\bar g)> M$. The strategy of the proof is to decompose $f$ into several pieces in order to apply the inductive hypothesis. To that end, we first \textit{collapse} the mean $\bar g$, by applying \cite[Lemma 3.8]{DS} to $T= \bar g$. For any $\epsilon \in (0,1)$, we obtain $S= \sum_{j=1}^J k_j \llbracket S_j \rrbracket \in \AQ$ which satisfies
	\begin{equation}\label{eq:clp1}
		\beta M \leq \beta d(\bar g) \leq s(S) < +\infty, 
	\end{equation} 
	\begin{equation}\label{eq:clp2}
		\GG(S, \bar g) \leq \epsilon s(S). 
	\end{equation} 
	Here $\beta = \beta(\epsilon, Q)$ is the constant in Lemma 3.8.
	The fact that $s(S)< +\infty$ means $J\geq 2$.  Recall \eqref{eq:ming} (this estimate is independent of the assumption on $d(\bar g)$), we get
	\begin{equation}\label{tmp:clp1}
		\min S \leq \min \bar g + \GG(S, \bar g) \leq C+ \epsilon s(S). 
	\end{equation} 
	Assume without loss of generality that $|S_1| = \min S$. We let
	\[ \widetilde S := k_1 \llbracket 0 \rrbracket + \sum_{j=2}^J k_j \llbracket S_j \rrbracket. \]
	By \eqref{tmp:clp1},
	\begin{equation}\label{tmp:clp2}
		\GG(S, \widetilde S) \leq \sqrt{k_1 |S_1|^2} < \sqrt{Q} \dist (\supp\, (S), 0) \leq C \sqrt{Q} + \epsilon \sqrt{Q} s(S).
	\end{equation} 
	We fix $\epsilon$ with $\epsilon\sqrt{Q} = \frac{1}{64}$; we may also choose $M = M(Q, \beta(\epsilon,Q))$ sufficiently large
	\begin{equation}\label{tmp:clp3}
		C  \leq \epsilon \beta M \leq \epsilon s(S). 
	\end{equation} 
	Thus it follows from \eqref{eq:clp1} that
	\[ \GG(S, \widetilde S)< 2\epsilon\sqrt{Q} s(S) = \frac{1}{32} s(S). \] 
	Combined with \eqref{eq:clp2}, we have
	\begin{equation}\label{eq:bargS}
		\GG(\bar g, \widetilde S) \leq \sqrt{2\GG(S, \bar g)^2 + 2 \GG(S, \widetilde S)^2} < \frac{1}{16} s(S). 
	\end{equation} 
	On the other hand, we also have 
	\begin{equation}\label{tmp:lbdist}
		s(\widetilde S) \geq (1-2\epsilon) s(S).
	\end{equation} 
	 In fact, either $s(\widetilde S) = |S_i - S_j|$ for $i, j \neq 1$, in which case $s(\widetilde S) \geq s(S)$ by definition; or $s(\widetilde S) = |S_i|$ for some $i \neq 1$, then by \eqref{tmp:clp1} and \eqref{tmp:clp3}
	\[ s(\widetilde S) = |S_i| \geq |S_i - S_1| - |S_1| \geq s(S) - \min S \geq (1-2\epsilon) s(S). \]
	
	Let 
	\[ \vartheta: \AQ \to \overline{B_{s(\widetilde S)/8}(\widetilde S)} \]
	be the retraction given by Lemma \ref{lm:retraction}. We define $h\in W^{1,2}(B_{1-\eta})$ by
	\[ h(x) := \vartheta \left( f\left( \frac{x}{1-\eta} \right) \right), \]
	where $\eta$ is a small parameter to be determined later. By Lemma \ref{lm:retraction} (iii), $h(x)$ contains a zero element for every $x\in B_{1-\eta}^-$. By removing one zero element in the lower half space we may consider $h$ as a function in $W^{1,2}(B_{1-\eta}, \AQt)$. Therefore by \cite[Theorem 4.2]{DDHM} there exists a Dir-minimizer $\hat h\in W^{1,2}(B_{1-\eta}, \AQt)$ with interface $(\gamma, \varphi)$, such that $\hat h = h$ on $\partial B_{1-\eta} \setminus \gamma$. Almost everywhere on $\partial B_{1-\eta}$, $\hat h$ takes value in $\overline{B_{s(\widetilde S)/8}(\widetilde S)}$. Therefore by Proposition \ref{prop:dcp} $\hat h$ can be decomposed into the sum of $h_1$ and $h_2$, where $h_1$ is a $K$-valued function and Dir-minimizer, $h_2$ is an $L$-valued function and Dir-minimizer with interface $(\gamma, \varphi)$, and $K, L \leq Q-1$. By \cite[Proposition 3.10]{DS} and the inductive hypothesis, we have
	\[ \Dir(h_1, B_{1-\eta}) \leq \left( \frac{1}{m-2} - \gamma_i \right)(1-\eta) \, \Dir(h_1, \partial B_{1-\eta}), \]
	\[ \Dir(h_2, B_{1-\eta}) \leq \left( \frac{1}{m-2} - \gamma_b \right)(1-\eta) \, \Dir(h_2, \partial B_{1-\eta}). \]
	Hence
	\begin{align*}
		\Dir(\hat h, B_{1-\eta}) & \leq  \left( \frac{1}{m-2} - \gamma \right)(1-\eta) \, \Dir(h, \partial B_{1-\eta}) \\
		&  = \left( \frac{1}{m-2} - \gamma \right)(1-\eta)^{m-2} \, \Dir(g, \partial B_{1}) \\
		& <   \left( \frac{1}{m-2} - \gamma \right).
	\end{align*} 
	Here $\gamma_0 = \min\{\gamma_i, \gamma_b\}>0$ is a constant depending on $m$ and $Q$.
	We consider the following competitor
	\[ \hat f = \left\{\begin{array}{ll}
		\hat h, & \text{ in } B_{1-\eta} \\
		\text{interpolation between } \hat h \text{ and } g \text{ as in Lemma \ref{lm:interpolation}, } & \text{ in } B_1 \setminus B_{1-\eta}.
	\end{array} \right. \]
	By the estimate \eqref{eq:interpolation},
	\begin{align}
		\Dir(\hat f, B_1) & \leq \Dir(\hat h, B_{1-\eta}) + C\eta\left( \Dir(\hat h, \partial B_{1-\eta}) + \Dir(g, \partial B_1) \right) + \frac{C}{\eta} \int_{\partial B_1} \GG(g, \vartheta(g))^2 \nonumber \\
		& < \frac{1}{m-2} - \gamma_0 +2C\eta + \frac{C}{\eta} \int_{\partial B_1} \GG(g, \vartheta(g))^2.\label{tmp:cpt}
	\end{align}
	
	Now we estimate the last term in the right hand side of \eqref{tmp:cpt}. By the definition of the retraction $\vartheta$, $g$ and $\vartheta(g)$ only differ on the set
	\[ E: = \left\{x\in \partial B_1: g(x) \notin \overline{B_{s(\widetilde S)/8}(\widetilde S) } \right\}. \]
	For every $x\in E$, by \eqref{eq:bargS} and the properties of $\vartheta$,
	\[ \GG(\vartheta \circ g(x), \bar g) = \GG(\vartheta \circ g(x), \vartheta( \bar g) ) <  \GG(g(x), \bar g), \]
	and
	\[ \GG(g(x), \bar g) \geq \GG(g(x), \widetilde S) - \GG(\bar g, \widetilde S) > \frac{1}{16} s(\widetilde S). \]
	Hence
	\begin{equation}\label{eq:bddiff}
		\int_{\partial B_1} \GG(g, \vartheta(g))^2 \leq 2\int_E \GG(g(x), \bar g)^2 + \GG\left(\vartheta\circ g(x), \bar g \right)^2 \leq 4 \int_E \GG(g(x), \bar g)^2 \leq C |E|^{\frac{2}{m-1} }.
	\end{equation}
	Recall \eqref{tmp:lbdist} and \eqref{tmp:clp3},
	\[ s(\widetilde S) \geq (1-2\epsilon) s(S) \geq (1-2\epsilon) \beta M. \]
	We may estimate the measure of $E$ by Chebyshev inequality
	\begin{align*}
		|E| \leq \int_{B_1} \left( \dfrac{\GG(g(x), \bar g) }{s(\widetilde S)/16 } \right)^2 \leq \frac{C}{M^2}.
	\end{align*}
	Combined with \eqref{tmp:cpt} and \eqref{eq:bddiff}, we conclude that
	\[ \Dir(\hat f, B_1) \leq \frac{1}{m-2} - \gamma_0 + C' \eta + \frac{C'}{\eta M^2}, \]
	where the constants $\gamma_0, C'$ only depend on $Q$ and $m$. We first choose $\eta$ so that $C'\eta = \frac{\gamma_0}{3}$, then we choose $M$ so that $\frac{C'}{\eta M^2} = \frac{\gamma_0}{3}$. Therefore by the minimality of $f$ 
	\[ \Dir(f, B_1) \leq \Dir(\hat f, B_1) \leq \frac{1}{m-2} - \frac{\gamma_0}{3}. \]
	
	\textit{Step 3. Conclusion.} With the value of $M$ fixed, Step 1 shows that if $d(\bar g) \leq M$, there exists $\gamma = \gamma(Q)>0$ such that
	\[ \Dir(f, B_1) \leq \frac{1}{m-2} - \gamma. \]
	Assuming the inductive hypothesis, Step 2 shows that if $d(\bar g) > M$,
	\[ \Dir(f, B_1) \leq \frac{1}{m-2} - \frac{\gamma_0}{3}. \]
	This concludes the proof.
\end{proof}
	
\subsection{Proof of Theorem \ref{thm:Holder}}
We want to prove the following decay of Dirichlet energy
\begin{equation}\label{eq:Dirbdrad}
	\Dir (f, B_r) \leq C r^{m-2+2\beta} (\Dir (f, B_1) + \|D\varphi\|_{C^0}^2)
\end{equation}
for every $y \in B_{\frac12}$ and almost every $0<r \leq \frac12$.

First of all observe that the estimate follows from Proposition \ref{prop:Dirbd} for $y\in \gamma$. 
Indeed in that case, if we let 
$h(r) = \iint_{B_r(y)} |Df|^2$, then $h$ is absolutely continuous and 
\[ h'(r) = \int_{\partial B_r(y)} |Df|^2 \geq \int_{\partial B_r(y)} |D_\tau f|^2 =: \Dir(f, \partial B_r(y)) \quad \text{ for almost every } r.
\]
Combined with \eqref{eq:Dirbd} we have
\[ 
h (r) \leq C(m) rh' (r) + C r^m \|D\varphi\|_{C^2}^2 \leq \frac{r h'(r)}{m-2+2\beta} + C r^m \|D\varphi\|_{C^0}^2\, ,
\]
(where $\beta$ is assumed to be smaller than $1$). 
We next define $k (r) := h (r) + A r^m$ and compute
\begin{align*}
k (r) & = h (r) + A r^m \leq \frac{r}{m-2+2\beta}  h' (r) + C \|D\varphi\|_{C^0}^2 r^m + A r^m\\
& \leq \frac{r}{m-2+2\beta} k'(r) + C \|D\varphi\|_{C^0}^2 r^m - A \left(\frac{m}{m-2+2\beta} -1\right) r^m\, .
\end{align*}
Since 
\[
\frac{m}{m-2+2\beta} -1 > 0\, ,
\]
for $A = C' \|D\varphi\|_{C_0}^2$ with $C'$ sufficiently large we conclude
\[
k (r) \leq \frac{r}{m-2+2\beta} k'(r)\, 
\]
and integrating the latter inequality in the interval $[r,1/2]$ we get the desired estimate
\begin{align*}
\Dir (f, B_r (y)) &\leq k (r) \leq r^{m-2+2\beta} k \left(\frac{1}{2}\right) 
\leq r^{m-2+2\beta} \left( \Dir (f, B_{1/2} (y)) + C' \|D\varphi\|_{C^0}^2\right)\\
&\leq C r^{m-2+2\beta} \left( \Dir (f, B_1)+ \|D\varphi\|_{C^0}^2\right)\, . 
\end{align*}

Consider now a point $y\in B_{1/2}\setminus \gamma$. If $r\geq \frac{1}{4}$ the estimate \eqref{eq:Dirbdrad} is then obvious. Hence we assume $r < \frac{1}{4}$. Let next $\rho:= \dist (y, \gamma)$. If $r\geq \rho$, consider $x\in \gamma$ such that $|x-y|= \dist (y, \gamma)$ and observe that $B_{2r} (x) \supset B_r (y)$. The estimate follows then from the one for $y\in \gamma$. Otherwise, we have two possibilities. If $\rho \geq \frac{1}{4} > r$, we then can use the decay estimate for $Q$-valued Dir-minimizers to infer
\[
\Dir (f, B_r (y)) \leq C r^{m-2+2\beta} \Dir (f, B_{1/4} (y)) \leq C r^{m-2+2\beta} \Dir (f, B_1)\, .
\]
If $r<\rho < \frac{1}{4}$ we can then proceed in two steps to prove
\[
\Dir (f, B_r (y)) \leq \left(\frac{r}{\rho}\right)^{m-2+2\beta} \Dir (f, B_\rho (y))
\leq C r^{m-2+2\beta} \left(\Dir (f, B_1) + \|D\varphi\|_{C^0}^2\right)\, .
\]
Having finally proved the decay \eqref{eq:Dirbdrad}, the H\"older continuity follows from the Campanato-Morrey estimate. 

\section{First variations and monotonicity of the frequency function}\label{s:frequency}

In this section we address a main tool to prove Theorem \ref{thm:main}, the monotonicity of the frequency function. The original frequency function was introduced by Almgren in \cite{Alm} for Dir-minimizing $Q$-valued map, cf. also \cite{DS}. The one for $(Q-\frac{1}{2})$-valued maps with interface $(\gamma, 0)$ in $\mathbb R^m$ was introduced in \cite{DDHM} and requires a subtle argument. Since our Theorem \ref{thm:main} is $2$-dimensional, we can take advantage of the reduction to Theorem \ref{thm:main_simple} and restrict our attention to the model situation in which the interface is $(\RR, 0)$. Under such assumption the statement and proof of the relevant formulae is just a straightforward adaptation of the arguments in \cite{DS}, which we give below for the reader's convenience (the issue in \cite{DDHM} is that in dimension $m\geq 3$ it is not possible to ``rectify'' a general $\gamma$ with a conformal change of coordinates). 

\begin{defn}[The frequency function]
Assume $f = (f^+, f^-)$ is a $(Q-\frac{1}{2})$-valued map on $\Omega \subset \RR^m$ with interface $(\gamma, \varphi)$ and consider a ball $B_r (x) \subset \Omega$ with $x\in \gamma$. 
	We define
	\begin{equation}
		D_{x,f}(r) = {\rm Dir}\, (f, B_r (x)) , \quad H_{x,f} (r) = \int_{\partial B_r (x)} |f|^2 := \int_{\partial B_r^+(x)}|f^+|^2 + \int_{\partial B_r^-(x)} |f^-|^2.
	\end{equation}
	When $H_{x,f}(r) > 0$, we define the frequency function
	\begin{equation}
		I_{x,f}(r) = \frac{r D_{x,f}(r)}{H_{x,f}(r)}.
	\end{equation}
\end{defn}
When $x$ and $f$ are clear from the context, we often use the shorthand notation $D(r), H(r)$ and $I(r)$.

\begin{prop}[First variations] Assume $f=(f^+, f^-) \in W^{1,2}(\Omega, \AQt)$ is Dir-minimizing on $\Omega \subset \RR^2$ with interface $(\RR, 0)$ and let $B_r \subset \Omega$. Then 
	\begin{equation}\label{eq:innervar}
\int_{\partial B_r(x)}|Df|^2 = 2 \left( \int_{\partial B_r^+(x)} \sum_{j=1}^Q |\partial_\nu f_j^+|^2 + \int_{\partial B_r^-(x)} \sum_{j=1 }^{Q-1} |\partial_\nu f_j^-|^2 \right)
	\end{equation}
and 
	\begin{equation}\label{eq:outervar}
		\iint_{B_r(x)} |Df|^2 = \int_{\partial B_r^+(x)} \sum_{j=1}^Q \langle \partial_\nu f_j^+, f_j^+ \rangle + \int_{\partial B_r^-(x)} \sum_{j=1 }^{Q-1} \langle \partial_\nu f_j^-, f_j^- \rangle\, .
	\end{equation}
	Here $\nu$ denotes the outer unit normal on the boundary of the given ball, and $f^+ = \sum_{j=1}^Q \llbracket f_j^+ \rrbracket $ and $f^- = \sum_{j=1}^{Q-1} \llbracket f_j^- \rrbracket$ are measurable selections of $f^+$ and $f^-$, given by \cite[Proposition 0.4]{DS}.
\end{prop}

\begin{remark}
Identity \eqref{eq:innervar} implies that the integral of the square of the tangential derivative on the circle $\partial B_r$ equals the integral of the square of the normal derivative.
\end{remark}

\begin{proof}
	The proof follows the same computations of \cite[Proof of  Proposition 3.2]{DS}. It just suffices to observe the following two facts:
\begin{itemize}
\item \eqref{eq:innervar} is derived by comparing the Dirichlet energy of $f$ with competitors of the form $f\circ \Phi_\varepsilon$, where $\{\Phi_\varepsilon\}$ are some specific one-parameter families of diffeomorphisms. It easy to check that the ones used in \cite[Proof of  Proposition 3.2]{DS} map $\RR$, $\Omega^+$ and $\Omega^-$ onto themselves and hence give an admissible family of competitors for our variational problem as well. 
\item Similarly, \eqref{eq:outervar} is derived by comparing the Dirichlet energy of $f$ with competitors of the form
\[
f^{\varepsilon, \pm} (x) := \sum_j \a{f_j^\pm (x) + \varepsilon \psi (x, f_j^\pm (x))} 
\]
where $\psi(x, u) = \phi(|x|) u$ satisfies $\psi(x, 0) = 0$. Therefore the functions $f^{\varepsilon, \pm}$ have also interface $(\RR, 0)$ and they are in the class of admissible competitors.\qedhere
\end{itemize}
\end{proof}

The above first variation formulae imply:
\begin{theorem}[Monotonicity of the frequency, analogue of Theorem 3.15 \cite{DS}]\label{thm:monot}
	Let $f$ be a $(Q-\frac{1}{2})$-valued Dir-minimizing map with interface $(\RR, 0)$ in an open set $\Omega\subset \RR^2$ containing the origin and assume that $f^+ (0) = Q \a{0}$. Either there exists $\delta>0$ such that 
	\[ f^+|_{B^+_\delta(0)} \equiv Q\llbracket 0 \rrbracket, \quad f^-|_{B^-_\delta(0)} \equiv (Q-1) \llbracket 0 \rrbracket; \] 
	or $I_{0,f}(r)$ is an absolutely continuous nondecreasing positive function on $(0,\dist (0, \partial\Omega))$.	
\end{theorem}

\begin{proof}
	If $H(r) = 0$ for some $r>0$, then $f^+=Q \a{0}$ a.e. on $\partial B^+_r(0)$ and $f^- = (Q-1) \a{0}$ a.e. on $\partial B^-_r (0)$. For such boundary data the only minimizer is the pair which is constant on the respective $B^\pm_r (0)$. From now on we assume therefore that $H(r)>0$ for every $r\in (0,1)$. 
	
	$D$ is absolutely continuous and
	\begin{equation}
		D'(r) = \int_{\partial B_r} |Df|^2 \text{ for almost every } r.
	\end{equation}
	Since $f^+, f^- \in W^{1,2}$ are approximately differentiable almost everywhere, we can apply the chain-rule formulas (see \cite[Propositions 1.12 and 2.8]{DS}) and justify the following computations:
	\begin{align}
		H'(r) & = \frac{d}{dr} \int_{\partial B_1^+ } r |f^+ (ry)|^2 dy + \frac{d}{dr} \int_{\partial B_1^- } r |f^- (ry)|^2 dy \nonumber \\
		 & = \int_{\partial B_1} |f(ry)|^2 dy + \int_{\partial B_1^+} r \frac{\partial}{\partial r} |f^+(ry)|^2 dy + \int_{\partial B_1^-} r \frac{\partial}{\partial r} |f^-(ry)|^2 dy \nonumber \\
		& = \frac{1}{r} \int_{\partial B_r} |f|^2 + 2 \int_{\partial B_r^+} \sum_{j=1}^Q \langle \partial_\nu f_j^+, f_j^+ \rangle + 2 \int_{\partial B_r^-} \sum_{j=1}^{Q-1} \langle \partial_\nu f_j^-, f_j^- \rangle
		= \frac{1}{r} H(r) + 2 D(r),\label{eq:Hder}
	\end{align}
	by the outer variation formula \eqref{eq:outervar}. In fact, since both $H(r)$ and $D(r)$ are continuous, we have $H\in C^1$ and the above inequality holds pointwise. Therefore
	\begin{align*}
		I'(r) & = \frac{D(r)}{H(r)} + \frac{rD'(r)}{H(r)} - rD(r) \frac{H'(r)}{H(r)^2} \\
		& = \frac{D(r)}{H(r)} + \frac{rD'(r)}{H(r)} - \frac{D(r)}{H(r)} - 2r \frac{D(r)^2}{H(r)^2} \\
		& = \frac{rD'(r)}{H(r)} - 2r \frac{D(r)^2}{H(r)^2}.
	\end{align*}
	Again by the inner and outer variations formulae \eqref{eq:innervar}, \eqref{eq:outervar}, we obtain
	\begin{align}
		I'(r)  = & \frac{2r}{H(r)^2} \left[ \left(\int_{\partial B^+_r} \sum_{j=1}^Q |\partial_\nu f_j^+|^2 + \int_{\partial B^-_r} \sum_{j=1}^{Q-1} |\partial_\nu f_j^-|^2  \right) \cdot \left( \int_{\partial B^+_r} \sum_{j=1}^Q | f_j^+|^2  + \int_{\partial B_r^-}\sum_{j=1}^{Q-1} | f_j^-|^2  \right) \right. \nonumber \\
		& \qquad \left. - \left( \int_{\partial B^+_r} \sum_{j=1}^Q \langle \partial_\nu f_j^+, f_j^+ \rangle + \int_{\partial B^-_r} 
\sum_{j=1}^{Q-1} \langle \partial_\nu f_j^-, f_j^- \rangle  \right)^2  \right]\, .
	\end{align}
We can now choose a measurable selection of the various multifuctions involved and extend such selections $f^\pm_j, \partial_\nu f^\pm_j$ to $0$ respectively on $B_r^\mp$. The Cauchy-Schwartz inequality will then imply:
\begin{align}
		I'(r)  = & \frac{2r}{H(r)^2} \left[ \int_{\partial B_r} \left(\sum_{j=1}^Q |\partial_\nu f_j^+|^2 + \sum_{j=1}^{Q-1} |\partial_\nu f_j^-|^2  \right) \cdot \int_{\partial B_r} \left( \sum_{j=1}^Q | f_j^+|^2  + \sum_{j=1}^{Q-1} | f_j^-|^2  \right) \right. \nonumber \\
		& \qquad \left. - \left( \int_{\partial B_r} \sum_{j=1}^Q \langle \partial_\nu f_j^+, f_j^+ \rangle + \sum_{j=1}^{Q-1} \langle \partial_\nu f_j^-, f_j^- \rangle  \right)^2  \right] \geq 0\, . \label{eq:monofreq}
	\end{align}
\end{proof}

\begin{corollary}\label{cor:hom}
	Let $f$ be as in Theorem \ref{thm:monot}.	$I_{0,f}(r) \equiv \alpha$ if and only if $(f^+, f^-)$ is $\alpha$-homogeneous, i.e.
	\[ f^{\pm} (x) = \sum_i \a{|x|^\alpha f^\pm_i \left(\frac{x}{|x|}\right)}\, . \]
\end{corollary}

In the rest of the note, when dealing with a $Q$-valued funcion $f = \sum_i \a{f_i}$, we will use the notation $\lambda f$ for the multifunction $\sum_i \a{\lambda f_i}$. Similarly, for a $(Q-\frac{1}{2})$-valued function $f= (f^+, f^-)$, $\lambda f$ will denote $(\lambda f^+, \lambda f^-)$. In particular, the homogeneity of $f$ in the corollary above will be expressed by the formula
\[
f (x) = |x|^\alpha f \left(\frac{x}{|x|}\right)\, .
\]

\begin{proof}
	Suppose $\alpha=0$. $I_{0,f}(r) \equiv 0$ if and only if each $f_j^{\pm}$ is constant, so clearly $(f^+, f^-)$ is $0$-homogeneous. If $(f^+, f^-)$ is $0$-homogeneous, then each $f_j^{\pm}$ is constant on the ray starting from the origin. Thus by the continuity of $f$ near the origin, each $f_j^{\pm}$ is constant on its domain and $I_{0,f}(r) \equiv 0$.
	
	Suppose $\alpha > 0$. Then by the proof of the above theorem, $I(r)$ is a constant if and only if equality occurs in \eqref{eq:monofreq}, i.e. if and only if there exists constants $\lambda_r \in \RR$ such that
	\[ \partial_\nu f_j^{\pm}(x) = \lambda_r f_j^{\pm}(x) \text{ for almost every } r \text{ and almost every } x \text{ with } |x| = r. \]
	Moreover,
	\[ \alpha = I(r) = \frac{rD(r)}{H(r)} = \frac{r\int_{\partial B_r} \sum \langle \partial_\nu f_j^{\pm}, f_j^{\pm} \rangle }{ \int_{\partial B_r} \sum |f_j^{\pm}|^2} = r \lambda_r. \]
	Therefore $I(r) \equiv \alpha$ if and only if
	\begin{equation}\label{eq:tmphom}
		\partial_\nu f_j^{\pm}(x) = \frac{\alpha}{|x|} f_j^{\pm}(x) \text{ for almost every } x.
	\end{equation} 
	If $f$ is $\alpha$-homogeneous, clearly \eqref{eq:tmphom} holds. On the other hand, suppose \eqref{eq:tmphom} holds, we want to show that $f$ is $\alpha$-homogeneous. To that end we let $x\in \partial B_1$ and $\sigma_x = \{rx: 0<  r\leq 1\}$ be the ray from the origin through $x$. Note that $f|_{\sigma_x}\in W^{1,2}$ for almost every $x\in \partial B_1$. For those $x$ \eqref{eq:tmphom} implies
	\[ \frac{d}{dr} \frac{f_j^{\pm}(rx)}{r^{\alpha}} \equiv 0. \]
	Thus $f_j^{\pm}(rx) = r^{\alpha} f_j^{\pm} (x)$ for all $0<r\leq 1$ and almost every $x\in \partial B_1$.
\end{proof}

\begin{corollary}[analogue of Corollary 3.18 \cite{DS}]
	Let $f$ be as in Theorem \ref{thm:monot}. Suppose $H(r) \neq 0$ for every $r\in (0, \dist (\partial \Omega , 0))$. Then
	\begin{enumerate}[(i)]
		\item for almost every $r < 1$,
			\begin{equation}\label{eq:tmpH}
				\frac{d}{dr}  \log \left( \frac{H(r)}{r^{m-1}} \right)  = \frac{2I(r)}{r },
			\end{equation}
			and thus for almost every $s\leq t < 1$,
			\begin{equation}\label{eq:tmpH2}
				\left( \frac{s}{t} \right)^{2I(t)} \frac{H(t)}{t^{m-1}} \leq \frac{H(s)}{s^{m-1}} \leq \left( \frac{s}{t} \right)^{2I(s)} \frac{H(t)}{t^{m-1}};
			\end{equation}
		\item for almost every $s\leq t<1$, if $I(t)>0$, then
			\begin{equation}\label{eq:Dratio}
				\frac{I(s)}{I(t)} \left(\frac{s}{t} \right)^{2I(t)} \frac{D(t)}{t^{m-2}} \leq \frac{D(s)}{s^{m-2}} \leq \left(\frac{s}{t} \right)^{2I(s)} \frac{D(t)}{t^{m-2}}.
			\end{equation}
	\end{enumerate}
\end{corollary}

\begin{proof}
	\eqref{eq:tmpH} follows from \eqref{eq:Hder}. \eqref{eq:tmpH2} follows from \eqref{eq:tmpH} and the monotonicity of the frequency function. Finally, \eqref{eq:Dratio}
	follows from \eqref{eq:tmpH2} and the definition of the frequency function.
\end{proof}

%

\section{Compactness and tangent functions in planar domains}\label{s:rate}

The monotonicity of the frequency function provides a way of studying the asymptotic behavior of a minimizer at small scales around a given point with highest multiplicity. In what follows we will often switch between different systems of coordinates in the plane. More precisely, a point $x\in \mathbb R^2$ will be identified with
\begin{itemize}
\item the pair $(x_1, x_2)$ which gives the standard Cartesian coordinates of $x$;
\item the complex number $z = x_1 + i x_2$;
\item the pair $(r, \theta)$ which gives the standard polar coordinates of $x$, namely $x_1 = r \cos \theta$, $x_2 = r \sin \theta$ and $z = r e^{i\theta}$.
\end{itemize}

\begin{defn}
Let $f$ be a Dir-minimizing $(Q-\frac{1}{2})$-valued map on a planar domain $\Omega$ with interface $(\RR, 0)$. Let $y$ be a point at the interface $\RR$ and assume that $\Dir(f,B_\rho (y)) >0$ for every $\rho$. We define the following rescalings of $f$ at $y$:
\begin{equation}\label{eq:defblowup}
	f_{y,\rho} (x) = \frac{f(\rho x+y)}{\sqrt{\Dir(f,B_\rho(y))}}.
\end{equation}
\end{defn}
The key point is that, up to subsequences, the latter rescalings converge locally strongly to nontrivial Dir-minimizers.

\begin{theorem}[Compactness, analogue of Theorem 3.19 in \cite{DS}]\label{thm:tangent}
	Let $f\in W^{1,2}(\Omega, \AQt)$ be a Dir-minimizing map on a planar domain $\Omega$ with interface $(\RR, 0)$. Assume $f^+(0) = Q\llbracket 0 \rrbracket$ and $\Dir(f, B_{\rho}) >0$ for every $\rho \in (0, \dist (0, \partial \Omega))$. Then for any sequence $\{f_{\rho_k}\}$ with $\rho_k \searrow 0$, a subsequence, not relabelled, converges locally uniformly to a function $g:\RR^2 \to \AQt$ satisfying the following properties:
	\begin{enumerate}[(a)]
		\item $\Dir(g,B_1) = 1$ and $g|_{U}$ is Dir-minimizing with interface $(\RR, 0)$ for any bounded set $U \subset \RR^2$;
		\item $g(x) = |x|^\alpha g\left( \frac{x}{|x|} \right)$, where $\alpha = I_{0,f}(0) >0$ is the frequency of $f$ at $0$.
	\end{enumerate}
\end{theorem}

From now, any limit of a sequence of rescalings $\{f_{\rho_k}\}_k$ with $\rho_k\downarrow 0$ will be called a {\em tangent function}. 
A feature of the $2$-dimensional setting is that the compactness theorem above can be considerably strengthened: analogously to the ``interior case'', cf. \cite[Theorem 5.3]{DS}, we can prove that the tangent function at a given point is unique and that the rescaling converge at a suitable rate to it. The key is to first show a suitable rate of convergence for the frequency function.

\begin{prop}[Rate of convergence, analogue of Proposition 5.2 in \cite{DS}]\label{prop:rate}
	Let $f$ be as in Theorem \ref{thm:tangent} and set $\alpha = I_{0,f}(0)$. 
	Then there exist constants $r_0, \beta, C, H_0, D_0 >0$ such that for every $0<r\leq r_0$,
	%
	\begin{equation}\label{eq:decayfr}
		0\leq I(r) - \alpha \leq C r^\beta,
	\end{equation}
	\begin{equation}\label{eq:decayDe}
		0\leq \frac{H(r)}{r^{2\alpha+1}} - H_0 \leq C r^\beta, \quad 0 \leq \frac{D(r)}{r^{2\alpha}} - D_0 \leq C r^\beta.
	\end{equation}
\end{prop}

\begin{theorem}[Unique tangent map, analogue of Theorem 5.3 \cite{DS}]\label{thm:cvtang} 
	Let $f\in W^{1,2}(\DD, \AQt)$ be as in Theorem \ref{thm:tangent} and denote by $\beta$ the exponent of the decay estimate \eqref{eq:decayfr}. Then the tangent function $f_0$ to $f$ at $0$ is unique and, moreover,
\begin{equation}\label{e:decay}
\|\mathcal{G} (f_{0, \rho}, f_0)\|^2_{L^2 (\mathbb S^1)}\leq C \rho^\beta\, .
\end{equation}
\end{theorem}

\subsection{Compactness and tangent functions: Proof of Theorem \ref{thm:tangent}}
	Let $B_R$ denote a ball of sufficiently large radius $R\gg 1$. By the definition \eqref{eq:defblowup}, 
	\[ \Dir(f_\rho, B_R) = \frac{D(\rho R)}{D(\rho)}  \leq R^{m-2+2I(\rho R)} \frac{I(\rho)}{I(\rho R)} \leq R^{m-2+2(\alpha+1)}, \]
	where we use the estimate \eqref{eq:Dratio} for the first inequality, and the properties of the frequency function for the second. Since $f_\rho$'s are all Dir-minimizing with interface $(\RR, 0)$, by Theorem \ref{thm:Holder} they are locally equi-H\"older continuous. The assumption $f^+(0) = Q\llbracket 0\rrbracket$ implies $f^+_\rho(0) = Q \llbracket 0 \rrbracket$ and $f^-_{\rho}(0) = (Q-1) \llbracket 0 \rrbracket$ for all $\rho$. Thus $f_\rho$'s are locally uniformly bounded, and by Arzel\`a-Ascoli Theorem a subsequence (not relabelled) converges uniformly on compact sets to a continuous function $g=(g^+, g^-)$. (To use Arzel\`a-Ascoli Theorem, we may add to $f^-_\rho$ a zero sheet to get functions valued in the metric space $\AQ$.) In particular $g^+|_{\gamma} = g^-|_{\gamma} + \llbracket 0 \rrbracket$; moreover $f^+_{\rho_k}$ converges weakly in $W^{1,2}_{loc} (\RR^2_+, \AQ)$ to $g^+$, and $f^-_{\rho_k}$ converges weakly in $W^{1,2}_{loc} (\RR^2_-, \mathcal{A}_{Q-1})$ to $g^-$ (see \cite[Definition 2.9]{DS}). By \eqref{eq:repD} it follows then that $\Dir(g, B_r) \leq \liminf_{k\to\infty} \Dir(f_{\rho_k}, B_r)$ for all $r>0$.
	
	\textit{Proof of (a)}. Let $R>0$ be fixed. We will show that for any $0<r\leq R$,
	\begin{equation}\label{eq:limitmin}
		\Dir(g, B_r) = \liminf_{k\to\infty} \Dir(f_{\rho_k}, B_r) \text{ and } g|_{B_r} \text{ is Dir-minimizing with interface } (\RR, 0).
	\end{equation}  
	For any $R>0$, we will show \eqref{eq:limitmin} holds for all $r\leq R$.
	
	
	By Fatou's Lemma,
	\[ \int_0^R \liminf_{k\to\infty} \Dir(f_{\rho_k},\partial B_r) \, dr \leq \liminf_{k\to\infty} \int_0^R \Dir(f_{\rho_k}, \partial B_r) \, dr \leq \liminf_{k\to\infty} \Dir(f_{\rho_k}, B_R) \leq C<+\infty. \]
	Hence for almost every $r\in (0,R)$, the integrand of the first term is finite, and moreover by weak convergence
	\begin{equation}\label{eq:Dbdshell}
		\Dir(g,\partial B_r) \leq \liminf_{k\to\infty} \Dir(f_{\rho_k},\partial B_r) \leq M < +\infty.
	\end{equation} 
	
	For the sake of contradiction, assume that either one of the statement in \eqref{eq:limitmin} fails for such $r$, then there exists a function $h\in W^{1,2}(B_r, \AQt)$ with interface $(\RR, 0)$ such that
	\begin{equation}\label{eq:tmpcontra}
		h|_{\partial B_r} = g|_{\partial B_r} \text{ and } \Dir(h, B_r) < \liminf_{k\to\infty} \Dir(f_{\rho_k}, B_r). 
	\end{equation} 
	Let $\delta = 1/N <1/2$ to be fixed later, and consider the functions $\widetilde f_{k}$ on $B_r$ defined by
	\[ \widetilde f_k = \left\{\begin{array}{ll}
		\left( \frac{1}{1-\delta} \right)^{\frac{m-2}{2}} h\left( \frac{x}{1-\delta} \right), & \text{ for } x\in B_{(1-\delta)r}, \\
		h_k(x), & \text{ for } x\in B_r \setminus B_{(1-\delta)r},
	\end{array} \right. \]
	where the $h_k$'s are the interpolation functions provided by Lemma \ref{lm:interpolation} between the maps $f_{\rho_k}\in W^{1,2}(\partial B_r, \AQt)$ and $h\left(\frac{x}{1-\delta} \right)\in W^{1,2}(\partial B_{(1-\delta)r}, \AQt)$. Notice that $h_k$'s satisfy the boundary condition $h_k^+|_{\gamma} = h_k^-|_{\gamma} + \llbracket 0 \rrbracket$. By the minimality of $f_{\rho_k}$, \eqref{eq:interpolation} and changes of variables, we have
	\begin{align*}
		\Dir(f_{\rho_k}, B_r) & \leq \Dir(\widetilde f_k, B_r) \\
		& \leq \Dir\left( \left( \frac{1}{1-\delta} \right)^{\frac{m-2}{2}} h\left( \frac{x}{1-\delta} \right), B_{(1-\delta)r} \right) + \Dir(h_k, B_r \setminus B_{(1-\delta)r}) \\
		& \leq \Dir(h, B_{r}) + C\delta r \Dir(f_{\rho_k}, \partial B_r) + C \delta r \Dir(h, \partial B_r) + \frac{C}{\delta} \sup_{x\in \partial B_r} \GG(h(x), f_{\rho_k}(x)) \\
		& \leq \Dir(h, B_r) + C\delta R \Dir(f_{\rho_k}, \partial B_r) + C\delta R \Dir(g, \partial B_r) + \frac{C}{\delta} \sup_{x\in \partial B_r} \GG(g(x), f_{\rho_k}(x)).
	\end{align*}
	Passing $k\to \infty$, by the uniform convergence of $f_{\rho_k}$ to $g$, \eqref{eq:Dbdshell} and the assumption \eqref{eq:tmpcontra}, we get
	\begin{equation}
		\limsup_{k\to\infty} \Dir(f_{\rho_k}, B_r) \leq \Dir(h,B_r) + 2C\delta RM < \liminf_{k\to \infty} \Dir(f_{\rho_k}, B_r) + 2C\delta RM.
	\end{equation}
	We get a contradiction by choosing $\delta$ arbitrarily small. Therefore \eqref{eq:limitmin} holds for almost every $r\in (0,R)$. By the upper semi-continuity of $\Dir(g, B_r)$ in $r$, it follows that \eqref{eq:limitmin} holds for all $r\leq R$.
	
	\textit{Proof of (b).} We observe that for every $r>0$,
	\[ I_g(r) = \frac{r \Dir(g, B_r)}{\int_{\partial B_r} |g|^2} = \liminf_{k\to \infty} \frac{r \Dir(f_{\rho_k}, B_r)}{\int_{\partial B_r} |f_{\rho_k}|^2} = \liminf_{k\to \infty} \frac{r \rho_k \Dir(f, B_{r \rho_k})}{\int_{\partial B_{r\rho_k}} |f|^2} = I_f(0). \]
	Since $g$ is Dir-minimizing, by Corollary \ref{cor:hom} it is a homogeneous function with homogenity $\alpha = I_f(0)$. If $\alpha=0$, a continuous $0$-homogeneous function with $g(0) = Q\llbracket 0 \rrbracket$ is necessarily $g \equiv Q\llbracket 0 \rrbracket$. This is in contradiction with $\Dir(g, B_1) = \lim_k \Dir (f_{\rho_k}, B_1) = 1$, and thus $\alpha>0$.

\subsection{Rate of convergence: Proof of Proposition \ref{prop:rate}}
	\textit{Step 1.} We claim the following estimate holds for some $\beta>0$:
	\begin{equation}\label{eq:IODE}
		I'(r) \geq \frac{2}{r} \left(\alpha + \beta - I(r) \right) (I(r) - \alpha).
	\end{equation}
	
	Recall \eqref{eq:Hder}, we have
	\begin{equation}
		I'(r) = \frac{rD'(r)}{H(r)} - \frac{2I^2(r)}{r}.
	\end{equation}
	Thus \eqref{eq:IODE} is reduced to prove
	\begin{equation}\label{eq:dcclaim}
		(2\alpha+\beta) D(r) \leq \frac{r D'(r)}{2} + \frac{\alpha(\alpha+\beta)H(r)}{r}.
	\end{equation}
	
	Let $r$ be fixed, and let $g(\theta) := f(r e^{i\theta})$. Consider the decomposition of $g(\theta)$ as in Proposition $g= g_0 + \sum_{j=1}^J g_j$, where $g_0 \in W^{1,2}(\So, \mathcal{A}_{Q_0}^{\pm})$ and $g_j \in W^{1,2}(\So, \mathcal{A}_{Q_j})$ are irreducible maps. Recall that for each irreducible $g_j$, we can find $\zeta_j\in W^{1,2}(\So, \RR^n)$ such that
	\begin{equation}
		g_j(\theta) = \sum_{i=1}^{Q_j} \left\llbracket \zeta_j\left( \frac{\theta + 2\pi i}{Q_j} \right) \right\rrbracket.
	\end{equation}
	We write the Fourier expansions of $\zeta_j$'s as
	\begin{equation}
		\zeta_j(\theta) = \frac{a_{j,0}}{2} + \sum_{l=1}^\infty \left(a_{j,l} \cos(l\theta) + b_{j,l} \sin(l\theta) \right), \quad \theta\in [0,2\pi].
	\end{equation} 
	Suppose $g_0$ unrolls to a function $\zeta_0: \So \to \RR^n$, as in Lemma \ref{lm:unrolling}. The boundary condition $g_0^+|_{\gamma} = g_0^-|_{\gamma} + \llbracket 0 \rrbracket$ implies $\zeta_0(0) = \zeta_0(2\pi) = 0$. Hence we write the Fourier expansion of $\zeta_0$ as
	\[ 
	\zeta_0(\theta) = \sum_{l=1}^\infty c_l \sin\left( \frac{ l\theta}{2} \right), \quad \theta \in [0,2\pi].
	\]
	Recall \eqref{eq:innervar} and Lemma \ref{lm:unrolling}, we get
	\begin{align}
		D'(r) = 2 \Dir(f, \partial B_r) & = \frac{2}{r} \Dir(g, \So) 
		= \frac{2}{r} \sum_{j=0}^{J} \Dir(g_j, \So) \nonumber \\
		& = \frac{2}{r} \left( \frac{2}{2Q_0-1} \Dir(\zeta_0, \So) + \sum_{j=1}^{J} \frac{1}{Q_j} \Dir(\zeta_j, \So) \right) \nonumber \\
		& = \frac{2\pi}{r} \left( \frac{1}{2(2Q_0-1)}   \sum_{l} c_l^2  l^2 + \sum_{j=1}^{J} \sum_l \frac{\left( a_{j,l}^2 + b_{j,l}^2 \right) l^2 }{Q_j}\right),\label{eq:dc1}
	\end{align}
	and 
	\begin{align}
		H(r) &= \int_{\partial B_r} |f|^2  = r \int_{\So} |g|^2 
		= r \left( \frac{2Q_0-1}{2} \int_{\So} |\zeta_0|^2 + \sum_{j=1}^J Q_j \int_{\So} |\zeta_{j}|^2 \right) \nonumber \\
		& = \pi r \left[ \frac{2Q_{0}-1}{2}  \sum_l  c_l^2   + \sum_{j=1}^J Q_j \left(  \frac{a_{j,0}^2}{2} + \sum_l \left( a_{j,l}^2 + b_{j,l}^2 \right) \right) \right]. \label{eq:dc2}
	\end{align}
	On the other hand, consider a $W^{1,2}$-extension of $\zeta_0$ on $B_1$ 	\begin{equation}\label{eq:bdsheetext}
		\overline \zeta_0(\rho, \theta) = \sum_{l=1}^\infty \rho^{\frac{l}{2}} c_l \sin \left( \frac{l\theta}{2} \right),		
	\end{equation} 
	(note that $\overline \zeta_0$ is not harmonic at the origin), and the harmonic extension of each $\zeta_j$:
	\[ \overline \zeta_j(\rho, \theta) = \frac{a_{j,0}}{2} + \sum_{l=1}^\infty \rho^l \left(a_{j,l} \cos(l\theta) + b_{j,l} \sin(l\theta) \right), \quad j=1, \cdots, J, \]
	where $ \rho\leq 1 \text{ and } \theta\in [0,2\pi]$.
	We then construct a competitor $h$ of $f$ in $B_r$ as follows: $h(\rho e^{i\theta}) = \widehat h(\rho e^{i\theta}/r)$, where $\widehat h$, a function defined on $B_1$, is given by
	\[
		\widehat h(\rho e^{i\theta}) := \left( h_0^+(\rho e^{i\theta} ), h_0^-(\rho e^{i\theta}) \right) + \sum_{j=1}^J \sum_{k=0}^{Q_j} \left\llbracket \overline \zeta_j \left( \rho^{\frac{1}{Q_j}}, \frac{ \theta+ 2\pi k}{Q_j} \right) \right\rrbracket,
	\]
	and $(h_0^+, h_0^-)$ is a $\left(Q_0 - \frac12 \right)$-valued function defined by
	\[
		h_0^+(\rho e^{i\theta}) = \overline \zeta_0 \left(\rho^{\frac{2}{2Q_0-1}}, \frac{2\theta}{2Q_0-1} + \frac{4\pi}{2Q_0-1} (k-1) \right), \quad \theta\in [0,\pi], k=1,\cdots,Q_0,
	\]
	\[
		h_0^-(\rho e^{i\theta}) = \overline \zeta_0 \left(\rho^{\frac{2}{2Q_0-1}}, \frac{2\theta}{2Q_0-1} + \frac{4\pi}{2Q_0-1} (k-1) \right), \quad \theta\in [\pi, 2\pi], k=1,\cdots,Q_0-1.
	\]
	Notice that by definition \eqref{eq:bdsheetext},
	\[ h_0^+|_{\gamma} = h_0^-|_{\gamma} + \llbracket 0 \rrbracket, \]
	hence $\widehat h$ has interface $(\gamma,\varphi)$.
	A simple computation, combined with Lemma \ref{lm:unrolling} and \cite[Lemma 3.12]{DS}, shows
	\begin{align*}
		\Dir(h, B_r) = \Dir(\widehat h, B_1) = \sum_{j=0}^J \iint_{B_1} |D \overline\zeta_j|^2 = \frac{\pi}{2} \sum_l c_l^2 l + \pi \sum_{j=1}^J \sum_l (a_{j,l}^2 + b_{j,l}^2 ) l.
	\end{align*}
	Thus by the minimality of $f$, we know
	\begin{equation}\label{eq:dc3}
		D(r) \leq \Dir(h, B_r) = \frac{\pi}{2} \sum_l c_l^2 l + \pi \sum_{j=1}^J \sum_l (a_{j,l}^2 + b_{j,l}^2 ) l.
	\end{equation}

	We combine \eqref{eq:dc1}, \eqref{eq:dc2} and \eqref{eq:dc3} and plug into the left and right hand sides of \eqref{eq:dcclaim}. After simplifications, we show that it is enough to find $\beta>0$ satisfying
	\begin{equation}
		\left[l-\alpha(2Q_0-1) \right] \left[ l-(\alpha+\beta)(2Q_0-1) \right] \geq 0,
	\end{equation}	
	\begin{equation}
		\left(l-\alpha Q_j \right) \left[ l-(\alpha+\beta)Q_j \right] \geq 0, \quad j=1, \cdots, J,
	\end{equation}	
	for every $l\in \NN$.
	This is equivalent to say the intervals $(\alpha(2Q_0-1), (\alpha+\beta)(2Q_0-1) )$ and $(\alpha Q_j, (\alpha+\beta) Q_j )$ do not contain integer points. This is verified, if we choose
	\[ \beta = \min_{1\leq k\leq Q}\left\{ \frac{\lfloor\alpha k \rfloor + 1 - \alpha k }{k}, \frac{\lfloor\alpha (2k-1) \rfloor + 1 - \alpha (2k-1) }{2k-1} \right\}>0. \]
	
	\textit{Step 2.}
	Since $I(r)$ is monotone decreasing, there exists $r_0>0$ such that $I(r) \leq \alpha + \frac{\beta}{2}$ for all $r\leq r_0$. Hence \eqref{eq:IODE} implies that
	\begin{equation}
		I'(r) \geq \frac{\beta}{r} \left( I(r) - \alpha \right) \text{ for almost every } r\leq r_0. 
	\end{equation}
	Integrating the differential inequality, we get the desired estimate
	\begin{equation}\label{eq:tmpIccl}
		I(r) - \alpha \leq \left( \frac{r}{r_0} \right)^\beta \left( I(r_0) - \alpha \right) \leq Cr^\beta. 
	\end{equation} 
	
	Recall that the derivative of $H$ satisfies \eqref{eq:Hder}. In particular when $m=2$ we have
	\[ \left( \frac{H(r)}{r} \right)' = \frac{2D(r)}{r}. \]
	This implies
	\begin{equation}\label{eq:Hderfiner}
		\left( \log \frac{H(r)}{r^{2\alpha+1}} \right)' = \left( \log \frac{H(r)}{r} - \log r^{2\alpha} \right)' = \left(\log \frac{H(r)}{r} \right)' - \frac{2\alpha}{r} = \frac{2}{r} \left( I(r) - \alpha \right) \geq 0.
	\end{equation}
	Hence $\frac{H(r)}{r^{2\alpha+1}}$ is monotone increasing. In particular, its limit exists as $r\to 0+$:
	\begin{equation}
		\frac{H(r)}{r^{2\alpha+1}} \geq \lim_{r\to 0} \frac{H(r)}{r^{2\alpha+1}} =: H_0 \geq 0.
	\end{equation}
	On the other hand combining \eqref{eq:Hderfiner} and \eqref{eq:tmpIccl} we get
	\[ \left( \log \frac{H(r)}{r^{2\alpha+1}} \right)' \leq 2C r^{\beta-1}, \quad \text{ thus} \left( \log \frac{H(r)e^{-C_{\beta} r^\beta} }{r^{2\alpha+1}} \right)' \leq 0.\]
	Hence $\frac{H(r)e^{-C_{\beta} r^\beta} }{r^{2\alpha+1}}$ is monotone decreasing, and
	\[ \frac{H(r)e^{-C_{\beta} r^\beta} }{r^{2\alpha+1}} \leq \lim_{r\to 0} \frac{H(r)e^{-C_{\beta} r^\beta} }{r^{2\alpha+1}} = \lim_{r\to 0} \frac{H(r) }{r^{2\alpha+1}} = H_0. \]
	In particular $H_0>0$. Moreover
	\[ \frac{H(r)}{r^{2\alpha+1}} \left( 1-C_\beta r^\beta \right) \leq \frac{H(r)e^{-C_{\beta} r^\beta} }{r^{2\alpha+1}} \leq H_0, \]
	and we conclude that
	\[ \frac{H(r)}{r^{2\alpha+1}} - H_0 \leq C_\beta \, \frac{H(r)}{r^{2\alpha+1}} \, r^\beta \leq C_\beta \frac{H(r_0)}{r_0^{2\alpha+1}} r^\beta \leq C r^\beta. \]
	The last inequality of \eqref{eq:decayDe} follows from:
	\[ \frac{D(r)}{ r^{2\alpha} } - D_0 = \left( I(r) - I_0 \right) \frac{H(r)}{r^{2\alpha+1}} + I_0 \left( \frac{H(r)}{ r^{2\alpha+1} } - H_0 \right), \]
	where $I_0 = \alpha$ and $D_0 = I_0 H_0$.

\subsection{Uniqueness of the tangent map: Proof of Theorem \ref{thm:cvtang}}
	Without loss of generality, we assume $D_0 = 1$. By the estimate \eqref{eq:decayDe} and the definition of blow-up, it follows that
	\begin{equation}\label{eq:blowupscal}
		f_{\rho}(r,\theta) = \rho^{-\alpha} f(r\rho,\theta) \left( 1+ O\left(\rho^{\beta/2}\right) \right).
	\end{equation}
	It suffices to show the existence of a uniform limit for the ``dominant'' function 
\[
h_{\rho}(r,\theta) = \rho^{-\alpha} f(r\rho,\theta)\, .
\] 
Note that the function $h_{\rho}$ is homogeneous:
	\[ h_{\rho}(r,\theta) = \rho^{-\alpha} f(r\rho,\theta) = r^{\alpha} h_{r\rho}(1,\theta), \]
	it is enough to prove the existence of a uniform limit for $h_{\rho}|_{\So}$. Each function
	\[ h_{\rho}|_{\So} = h_{\rho}(1,\theta) = \rho^{-\alpha} f(\rho,\theta) \]
	is Dir-minimizing, so by Theorem \ref{thm:Holder} it is H\"older continuous with a uniform constant. We first show that $h_{\rho}|_{\So}$ has an $L^2$ limit.
	
	Let $\{T_i\}$ and $\{T'_i\}$ be countable dense subsets of $\AQ(\RR^n)$ and $\mathcal{A}_{Q-1}(\RR^n)$ respectively.
	We consider $r/2 \leq s\leq r$ and estimate
	\begin{align*}
		\int_0^{2\pi} \GG(h_r,h_s)^2 d\theta & = \int_0^{2\pi} \GG \left( \frac{f(r,\theta)}{r^{\alpha}}, \frac{f(s,\theta)}{s^{\alpha}} \right)^2 d\theta \\
		& = \int_0^{\pi} \sup_i \left| \GG \left( \frac{f^+(r,\theta)}{r^{\alpha}}, T_i \right) - \GG \left( \frac{f^+(s,\theta)}{s^{\alpha}}, T_i \right) \right| ^2 d\theta \, +  \\
		& \qquad \int_{\pi}^{2\pi} \sup_i \left| \GG \left( \frac{f^-(r,\theta)}{r^{\alpha}}, T'_i \right) - \GG \left( \frac{f^-(s,\theta)}{s^{\alpha}}, T'_i \right) \right| ^2 d\theta \\
		& \leq (r-s) \int_0^{\pi} \sup_i  \int_s^r \left|\frac{\partial}{\partial t} \GG \left(\frac{f^+(t,\theta)}{t^{\alpha}}, T_i \right) \right|^2 \, dt \,  d\theta \, + \\
		& \qquad (r-s) \int_{\pi}^{2\pi} \sup_i  \int_s^r \left|\frac{\partial}{\partial t} \GG \left(\frac{f^-(t,\theta)}{t^{\alpha}}, T'_i \right) \right|^2 \, dt \,  d\theta \\
		& \leq (r-s) \int_0^{\pi}  \int_s^r \left|\frac{\partial}{\partial t} \left(\frac{f^+(t,\theta)}{t^{\alpha}} \right) \right|^2 \, dt \,  d\theta \, +\\
&\qquad (r-s) \int_0^{\pi}  \int_s^r \left|\frac{\partial}{\partial t} \left(\frac{f^-(t,\theta)}{t^{\alpha}} \right) \right|^2 \, dt \,  d\theta \\
		& = (r-s) \int_0^{\pi} \int_s^r \sum_j \left\{ \alpha^2 \frac{|f^+_j|^2}{t^{2\alpha+2}} + \frac{|\partial_\nu f^+_j|^2}{t^{2\alpha}} - 2\alpha \frac{\langle \partial_\nu f^+_j, f^+_j \rangle}{ t^{2\alpha+1}} \right\} \, dt\, d\theta \, + \\
		& \qquad (r-s) \int_0^{\pi} \int_s^r \sum_j \left\{ \alpha^2 \frac{|f^-_j|^2}{t^{2\alpha+2}} + \frac{|\partial_\nu f^-_j|^2}{t^{2\alpha}} - 2\alpha \frac{\langle \partial_\nu f^-_j, f^-_j \rangle}{ t^{2\alpha+1}} \right\} \, dt\, d\theta \\
		& = (r-s) \int_s^r \left\{ \alpha^2 \frac{H(t)}{t^{2\alpha+3}} + \frac{D'(t)}{2t^{2\alpha+1}} - 2\alpha \frac{D(t)}{t^{2\alpha+2}} \right\} dt \\
		& = (r-s) \int_s^t \left\{ \frac{1}{2t} \left( \frac{D(t)}{t^{2\alpha}} \right)' + \alpha \frac{H(t)}{2t^{2\alpha+3}} \left(\alpha - I_{0,f}(t) \right) \right\} dt \\
		& \leq (r-s) \int_s^t  \frac{1}{2t} \left( \frac{D(t)}{t^{2\alpha}} - D_0 \right)' dt\,  \\
		& = (r-s) \left[ \frac{1}{2r} \left( \frac{D(r)}{r^{2\alpha}} - D_0 \right) - \frac{1}{2s} \left( \frac{D(s)}{s^{2\alpha}} - D_0 \right) \right] + \\
 &\qquad (r-s) \int_s^t  \frac{1}{2t^2} \left( \frac{D(t)}{t^{2\alpha}} - D_0 \right) dt \\
		& \leq Cr^{\beta}.
	\end{align*}
	Let $s\leq r$ be arbitrary, and let $l$ be a positive integer such that $r/2^{l+1} <s \leq r/2^l$. Iterating the above estimate we get
	\begin{align*}
		\|\GG(h_r,h_s)\|_{L^2(\So)} \leq \sum_{k=0}^l \|\GG\left(h_{r/2^k}, h_{r/2^{k+1}} \right) \|_{L^2(\So)} + \|\GG(h_{r/2^l}, h_s) \|_{L^2(\So)}   \leq \sum_{k=0}^l C\left( \frac{r}{2^k} \right)^{\frac{\beta}{2}} \leq C' r^{\frac{\beta}{2}}.
	\end{align*}
	This shows that $h_{\rho}|_{\So}$ is a Cauchy sequence in $L^2$, and thus converges to a limit function $h\in L^2(\So)$. Moreover, since $h_{\rho}$ is equi-H\"older continuous on $\So$, it follows that $h_{\rho}$ converges uniformly to $h$ on $\So$. In other words, $f_{\rho}$ converges locally uniformly to an $\alpha$-homogeneous function $g$, with $g(z) = |z|^\alpha h\left( \frac{z}{|z|} \right)$.

\section{Homogeneous Dir-minimizers}\label{s:classification}

In this section we study homogeneous Dir-minimizers in planar domains. We do not really give a complete characterization, but rather a set of necessary conditions that they have to satisfy. However, we will show below that each ``irreducible homogeneous piece'' in the classification is in fact minimizing. 

\begin{prop}[Characterization of tangent maps]\label{prop:chartang}
	Let $\alpha >0$ and let $f \in W^{1,2}_{loc} (\RR^2, \AQt)$ be a nontrivial $\alpha$-homogeneous Dir-minimizer with interface $(\RR, 0)$. 
Then the following alternatives hold:
	\begin{itemize}
	\item[(a)] either $\alpha = l$ for some $l\in \mathbb N$ and 
	\[
	\begin{split}
		f^+ & = k_0 \a{r^l \vec{c} \sin(l\theta)} + \sum_{j=1}^{J} k_j \left\llbracket r^{l} \left( \vec{a_j} \cos(l\theta) + \vec{b_j} \sin(l\theta) \right) \right\rrbracket =: f_0^+ + \sum_{j=1}^J k_j f_j, \\
		f^- & = (k_0-1) \a{r^l \vec{c} \sin(l\theta)} + \sum_{j=1}^{J} k_j \left\llbracket r^{l} \left( \vec{a_j} \cos(l\theta) + \vec{b_j} \sin(l\theta) \right) \right\rrbracket=:  f_0^- + \sum_{j=1}^J k_j f_j;
	\end{split}
	\]
\item[(b)] or $\alpha = \frac{n^*}{Q^*}$ for some coprime natural numbers $n^*, Q^*$ and 
\[
\begin{split}
		f^+ & = k_0 \a{0} + \sum_{j=1}^{J} k_j \sum_{z^{Q^*} = x \atop z=(r,\theta)} \left\llbracket r^{n^*} \left( \vec{a_j} \cos(n^* \theta) + \vec{b_j} \sin(n^* \theta) \right) \right\rrbracket =: f_0^+ + \sum_{j=1}^J k_j f_j, \\
		f^- & = (k_0-1) \a{0} + \sum_{j=1}^{J} k_j \sum_{z^{Q^*} = x \atop z=(r,\theta)} \left\llbracket r^{n^*} \left( \vec{a_j} \cos(n^* \theta) + \vec{b_j} \sin(n^* \theta) \right) \right\rrbracket =:  f_0^- + \sum_{j=1}^J k_j f_j\, ;
	\end{split}
\]
	\item[(c)] or $\alpha = \frac{2}{3}$ and
	\[
	\begin{split}
		f^+ & = \a{r^{\frac23} \vec{c} \sin \left( {\textstyle{\frac23}} \theta \right)} + \a{r^{\frac23} \vec{c} \sin \left( {\textstyle{\frac23}} (\theta + 2\pi) \right)} + \sum_{j=1}^J k_j \sum_{z^3 = x \atop z=(r,\theta)} \left\llbracket r^2 \left( \vec{a_j} \cos(2\theta) + \vec{b_j} \sin(2\theta) \right) \right\rrbracket \\
		& =: f_0^+ + \sum_{j=1}^J k_j f_j, \\
		f^- & = \a{r^{\frac23} \vec{c} \sin \left( {\textstyle{\frac23 \theta}} \right)}  + \sum_{j=1}^J k_j \sum_{z^3 = x \atop z=(r,\theta)} \left\llbracket r^2 \left( \vec{a_j} \cos(2\theta) + \vec{b_j} \sin(2\theta) \right) \right\rrbracket \\
		& =: f_0^- + \sum_{j=1}^J k_j f_j\, .
	\end{split}
	\]
	\end{itemize}
In all three cases the supports of $f_i(x)$ and $f_j(x)$ are disjoint for any $i\neq j\in \{0, 1, \cdots, J\}$ and any $x$ not at the origin. The constant vector $\vec{c} \neq 0$ and the pair of vectors $\vec{a}_j$ and $\vec{b}_j$ are linearly independent.

Finally, under the additional assumptions that $f$ satisfies \eqref{eq:avgsym} and its Dirichlet energy is positive, in the cases (a) and (b) the portions $\sum_j k_j f_j$ must necessarily be nontrivial, namely $J\geq 1$ and we cannot have $f = (f_0^+, f_0^-)$.
\end{prop}

The proof of Proposition \ref{prop:chartang} is based on the following characterization of homogeneous {\em irreducible} Dirichlet minimizing $(Q_0-\frac{1}{2})$-maps with interface $(\mathbb R, 0)$. Note in particular that, in case (b) of the Proposition above, the trace of the corresponding homogeneous map on the circle is reducible, with the only exception of the very trivial map $f_0^+ = \a{0}$.

\begin{lemma}\label{l:char_irr_hom}
Let $g_0 \in W^{1,2} (\So, \AQt)$ with $Q= Q_0$ be irreducible and assume that it is the trace of a non-trivial homogeneous Dir-minimizing map $f_0$ on $\So$ with homogeneity $\alpha$.   
Then:
\begin{itemize}
\item[(i)] Either $Q_0=1$, $\alpha = l \in \mathbb N\setminus \{0\}$ and $f_0^+ (r, \theta) = \vec{c} r^l \sin (l \theta)$ is a classical homogeneous harmonic polynomial with trace $0$ on $\mathbb R$;
\item[(ii)] Or $Q_0 =2$, $\alpha = \frac{2}{3}$ and $f_0 = (f_0^+, f_0^-)$, where $f_0^+$ and $f_0^-$ are the maps of Proposition \ref{prop:chartang}(c). 
\end{itemize}
\end{lemma}

\begin{remark}\label{r:minimality_of_examples} Clearly, since they are classical harmonic functions, all the examples in case (i) are actual Dirichlet minimizers. It is much less obvious that the examples in (ii) are also minimizers. This is not really needed in the proof of our main result. However, an elementary argument, which we include for completeness at the end of the paper in Proposition \ref{p:exception}, shows that in fact they are.  
\end{remark}

\begin{proof}[Proof of Lemma \ref{l:char_irr_hom}] 
By Proposition \ref{prop:decomp} there exists a function $\zeta:\So \to \RR^n$ satisfying $\zeta(0) = 0$ such that $g_0$ unwinds to $\zeta$. Let 
	\[ \zeta(\rho, \theta) = \rho^{\alpha \frac{2Q_0-1}{2}} \zeta(\theta) \]
	be an extension of $\zeta$ to the disk $\DD$. By \eqref{eq:unwindsld1} and \eqref{eq:unwindsld2} $f_0$ unwinds to $\zeta(\rho, \theta)$, and thus 
\[
\Dir(f_0, \DD) = \iint_{\DD} |D\zeta|^2 
\] 
by Lemma \ref{lm:unrolling}.
	 We consider the function $\overline \zeta(z) := \zeta(z^2): \DD^+ \to \RR^n $. By definition $\overline \zeta|_{\mathbb R} \equiv 0$. Since conformal maps do not change Dirichlet energy, we have
	 \begin{equation}\label{eq:tmpchartg}
	 	\iint_{\DD^+} |D\overline \zeta|^2 = \iint_{\DD} |D\zeta|^2 = \Dir(f_0, \DD).
	 \end{equation}
	 Consider any function $\eta: \DD^+ \to \RR^n$ satisfying $\eta|_{\partial \DD^+} = \overline \zeta|_{\partial \DD^+}$, 
we can wind the function $\eta(\sqrt{z}): \DD \to \RR^n$ by the formula \eqref{eq:unwindsld1}, \eqref{eq:unwindsld2} (where we use a branch of the square root function, for instance setting $z= r e^{i\theta}$ with $\theta \in [0, 2\pi[$ and defining $\sqrt{r e^{i\theta}} = r^{\frac{1}{2}} e^{i\theta/2}$) and find a corresponding function $h: \DD \to \AQt$ such that $h|_{\partial \DD} = f|_{\partial \DD}$ and $h^+|_{\mathbb R} = h^-|_{\mathbb R} + \llbracket 0 \rrbracket$. By the minimality of $f$ with interface $(\mathbb R, 0)$, we have
	 \[ \Dir(f_0,\DD) \leq \Dir(h, \DD) = \iint_{\DD^+} |D\eta|^2. \]
	 This combined with \eqref{eq:tmpchartg} shows that $\overline \zeta$ is a Dir-minimizer in $\DD^+$ which equals $0$ on $\mathbb R$. Thus $\overline \zeta$ is a harmonic function in $\DD^+$ which by Schwarz reflection can be extend it to $\DD$. On the other hand $\overline \zeta$ is $\alpha(2Q_0-1)$-homogeneous. By spherical harmonics we know $\alpha(2Q_0-1) = l\in \NN$ and $\overline \zeta(r, \theta) = \vec{c} r^l \sin(l\theta)$ with some constant $\vec{c}\in \RR^n$. Therefore $\zeta(\theta) = \vec{c} \sin\left( \frac{l\theta}{2} \right)$ on $\So$. 
	 
We now claim that, since the $\left(Q_0 -\frac12 \right)$-valued map $g_0$ unwinds to $\zeta(\theta) = \vec{c} \sin\left( \frac{l\theta}{2} \right)$ on $\So$ and it is irreducible, then either $Q_0=1$ or $l=Q_0=2$. In the first case $g_0^+ = \vec{c} \sin(l\theta)$ for some integer $l\in \NN\setminus \{0\}$ and we fall in the first case of the classification. In the second case 
	 \begin{equation}\label{eq:g0u}
	 	g_0^+ = \left\llbracket \vec{c} \sin\left( \frac23 \theta \right) \right\rrbracket + \left\llbracket \vec{c}\sin\left( \frac23 (\theta + 2\pi) \right) \right \rrbracket, \quad \theta \in [0,\pi], 
	 \end{equation} 
	 \begin{equation}\label{eq:g0d}
	 	g_0^- = \left\llbracket \vec{c}\sin\left( \frac23 \theta \right) \right\rrbracket, \quad \theta \in [\pi, 2\pi]\, ,
	 \end{equation} 
which covers the second case of the classification.

We next show our claim.	 
	 When $Q_0 = 1$, the condition (i) in Proposition \ref{prop:decomp} holds trivially, thus $g_0$ is irreducible. Now assume $Q_0>1$. The condition (i) fails if we can find $\theta\in [0,2\pi]$ and $k\in \NN$ such that
	 \begin{equation}\label{eq:charirred}
	 	\zeta(\theta) = \zeta\left( \theta + \frac{4\pi}{2Q_0-1} k \right),\quad 0\leq \theta, \theta + \frac{4\pi}{2Q_0-1} k < 2\pi. 
	 \end{equation} 
	 We denote $\beta = l\theta/2$, then \eqref{eq:charirred} becomes 
	 \begin{equation}\label{eq:contrairred}
	 	\sin\left( \beta \right) = \sin \left( \beta + \frac{2\pi lk}{2Q_0-1} \right), \quad 0 \leq \beta, \beta + \frac{2\pi l k}{2Q_0-1} < l \pi. 
	 \end{equation} 
	 To rephrase it slightly different, \eqref{eq:contrairred} is equivalent to find $\beta_1, \beta_2 \in [0, l\pi)$ such that $ \sin(\beta_1) = \sin(\beta_2)$ and $\beta_1 - \beta_2 = \frac{2 k}{2Q_0-1} l\pi$ for some $k\in \NN\setminus \{0\}$. For all odd integers $l \in \NN$, we can always find $\beta_1, \beta_2 \in [0,l\pi)$ with arbitrary distance in the range $[0,l\pi)$ satisfying $\sin(\beta_1) = \sin(\beta_2)$; for all even integers $l \in \NN$, we can always find $\beta_1, \beta_2 \in [0,l\pi)$ having the same sinus and with arbitrary distance in the range $[0, (l-1)\pi]$. In the latter case, the only way (i) could be satisfied is that there is an even integer $l\geq 2$ so that
	 \[ \frac{2}{2Q_0-1} l\pi > (l-1)\pi\, . \]
Namely we are looking for even numbers $l\geq 2$ and natural numbers $Q_0 >1$ such that $\frac{2}{2Q_0-1} > \frac{l-1}{l}$. Clearly, $\frac{l-1}{l} \geq \frac{1}{2}$. When $Q_0\geq 3$ we have $\frac{2}{2Q_0-1} \leq \frac{2}{5} < \frac{1}{2}$. Hence $Q_0=2$ is the only possibility: in that case $\frac{2}{2Q_0 -1} = \frac{2}{3}$. Then $l=2$ satisfies the inequality $\frac{l-1}{l} < \frac{2}{2Q_0-1}$, but as soon as $l\geq 4$ we have $\frac{l-1}{l} \geq \frac{3}{4} > \frac{2}{3}$. This restricts the possibilities to the only case $Q_0=l=2$ and thus proves our claim. 
\end{proof}

\begin{proof}[Proof of Proposition \ref{prop:chartang}] Denote by $\alpha$ the homogeneity of the map $f$ and let $g = f|_\So$. 

	We decompose $g\in W^{1,2}(\So, \AQt)$ into irreducible pieces as described in Proposition \ref{prop:decomp}:
	\[ g = \tilde g_0 + \tilde g = \tilde g_0 + \tilde{g} \] 
	where $\tilde g_0 = (\tilde g^+_0, \tilde g^-_0) \in W^{1,2}(\So, \mathcal{A}_{Q_0}^{\pm})$ is an irreducible map with interface $(\mathbb R, 0)$. We then get a similar decomposition of $f$ as $\tilde f_0 + \tilde f$, where both are $\alpha$-homogeneous and Dir-minimizing. 

Recall that according to Lemma \ref{l:char_irr_hom}, either $Q_0 =1$ and so $f_0^+$ is a classical harmonic polynomial with trace $0$ on $\mathbb R$, or $Q_0 =2$, in which case $f_0$ is the map of case (c).
According to \cite[Proposition 5.1]{DS} we can further decompose
\[
\tilde{f} = \sum_j k_j \tilde{f_j}
\]
where the traces $\tilde{g}_j$ of each $\tilde{f}_j$ on $\So$ are irreducible and have disjoint supports, namely 
\begin{equation}\label{e:non-intersection-1}
{\rm spt} (\tilde{g}_j (\theta)) \cap {\rm spt} (\tilde{g}_i (\theta)) = \emptyset\qquad
\mbox{for every $\theta$.}
\end{equation}
We next wish to show that
\begin{equation}\label{e:non-intersection-2}
\mbox{if ${\rm spt} (\tilde{g}^\pm_0 (\theta)) \cap {\rm spt} (\tilde{g}_j (\theta)) \neq \emptyset$ then $Q_0=1$ and $\tilde{g}_j (\theta) = \tilde{g}^+_0 (\theta)\;\; \forall \theta \in [0, \pi]$,} 
\end{equation}
namely $\tilde{f}_j$ is a classical harmonic function which coincides with $\tilde{f}_0^+$ on the upper half plane and can thus be obtained by Schwarz reflection on the lower half plane.
We argue for \eqref{e:non-intersection-2} and we distinguish two cases:

\medskip

{\bf Case $Q_0=2$.} $\tilde{f}_j$ must have homogeneity $\frac{2}{3}$. By the classification result of \cite[Proposition 5.1]{DS} it must take the form
\[
\sum_{z^3 = x \atop z=(r,\theta)} \left\llbracket r^2 \left( \vec{a_j} \cos(2\theta) + \vec{b_j} \sin(2\theta) \right) \right\rrbracket\, .
\]
where $\vec{a_j}$ and $\vec{b}_j$ span a $2$-dimensional plane, or it must the trivial map $\tilde{f}_j = \a{0}$. 

Assume that for some $\bar \theta$ we have ${\rm spt} (\tilde{g}^+_0 (\bar \theta))\cap {\rm \spt} (\tilde{g}_j (\bar \theta))\neq \emptyset$. If $\bar \theta \in ]0, \pi[$, then either we have a line of ``interior singularities'', contradicting the regularity theory of \cite{DS}, or ${\rm spt} (\tilde{g}^+_0 (\theta))\subset {\rm spt} (\tilde{g}_j (\theta))$ $\forall \theta \in ]0, \pi[$. The latter condition would however contradict the linear independence of $\vec{a}_j$ and $\vec{b}_j$. A similar argument can be used to exclude that there is any intersection between ${\rm spt}\, (\tilde{g}_0^- (\bar\theta))$ and ${\rm spt} (\tilde{g}_j (\bar\theta))$ when $\bar\theta\in ]\pi, 2\pi[$.

If $\bar\theta = 0$, then 
${\rm spt}\, (\tilde{g}_j (0))$ contains $0$ or $\vec{c} \sin \frac{4\pi}{3}\neq 0$. Like above, the second possibility (and the continuity at the interface) would imply ${\rm spt} (\tilde{g}^+_0 (\theta))\subset {\rm spt} (\tilde{g}_j (\theta))$ $\forall \theta \in [0, \pi]$ by the interior regularity theory of \cite{DS}. 
If instead $0 \in \spt \tilde{g}_j (0)$, $\tilde{g}_j$ would have to be trivial, because in the other alternative $\vec{a}_j$ and $\vec{b}_j$ must be linearly independent.  Hence $\tilde{f}_j$ vanishes identically. This being the case, consider the $5/2$-valued map $\a{\tilde{f}_0} + \a{\tilde{f}_j}$, which must be Dir-minimizing, and decompose it differently by introducing the maps
\[
h (r, \theta) = \left\{\begin{array}{ll}
\tilde{f}^+_0 (r, \theta) \qquad &\mbox{ if $\theta \in [0, \pi]$}\\
\tilde{f}_0^- (\theta) + \tilde{f}_j (r, \theta) \qquad &\mbox{ if $\theta \in [\pi, 2\pi]$.}
\end{array}\right.
\]
and 
\[
h^+ (r, \theta) = \tilde{f}_j (r, \theta) \qquad \mbox{if $\theta \in [0, \pi]$.}
\]
Clearly $h^+ + h$ is the same $\frac{5}{2}$-valued Dir minimizer and $h$ must therefore be a $2$-valued Dir-minimizing map. It is however $\frac{2}{3}$ homogeneous, but not of the form given in \cite[Proposition 5.1]{DS}, which is a contradiction.

Since we can argue similarly for $\bar \theta = \pi$, this shows that when $Q_0=2$ the supports of $\tilde{g}_0^\pm (\theta)$ and $\tilde{g}_j (\theta)$ must be disjoint for all $\theta$.

\medskip

{\bf Case $Q_0=1$.} Here we only have to examine $\bar \theta \in [0, \pi]$.
If ${\rm spt} (\tilde{g}_j (\bar \theta))$ intersects ${\rm spt} (\tilde g^+_0 (\bar \theta))$ for $\bar \theta \in ]0, \pi[$, then arguing as above, the interior regularity theory would imply ${\rm spt}\, (\tilde{g}^+_0 (\theta)) \subset {\rm spt}\, (\tilde{g}_j (\theta))$ for all $\theta \in [0, \pi]$. But then the homogeneity of $\tilde{f}_j$ must be an integer and $\tilde{f}_j$ is a single-valued classical harmonic polynomial by \cite[Proposition 5.1]{DS}. Since on the upper half plane such polynomial coincides with $\tilde{f}_0^+$, on the lower half plane it is determined by Schwarz reflection. 

If $\bar \theta \in \{0, \pi\}$, then either $0\in \spt (\tilde g_j(0))$ or $0\in \spt (\tilde g_j(\pi))$. If the homogeneity of $\tilde{f}_j$ is integral, then it is a classical harmonic function, which thus vanishes identically on $\mathbb R$. We then argue as above and consider the $3/2$-valued map $\a{\tilde{f}_0} + \a{\tilde{f}_j}$, which must be Dir-minimizing. As above, wecompose it differently by introducing the maps
\[
h (r, \theta) = \left\{\begin{array}{ll}
\tilde{f}^+_0 (r, \theta) \qquad &\mbox{ if $\theta \in [0, \pi]$}\\
\tilde{f}_j (r, \theta) \qquad &\mbox{ if $\theta \in [\pi, 2\pi]$.}
\end{array}\right.
\]
and 
\[
h^+ (r, \theta) = \tilde{f}_j (r, \theta) \qquad \mbox{if $\theta \in [0, \pi]$.}
\]
Clearly $h^+ + h$ is the same $\frac{3}{2}$-valued Dir minimizer and $h$ must therefore be harmonic. By the Schwarz reflection principle, we conclude that in fact $\tilde{f}_0^+$ is the restriction of $\tilde{f}_j$ on the upper half plane, namely it vanishes identically.

We thus must finally examine the possibility that $0\in \spt (\tilde g_j(0))$ or $0\in \spt (\tilde g_j(\pi))$ and the homogeneity is not a natural number, namely
$\alpha = \frac{n^*}{Q^*}$ where $n^*$ and $Q^*$ are coprime. The classification of \cite[Proposition 5.1]{DS} implies that either
\[
\tilde{f}_j (x) = \sum_{z^{Q^*} = x \atop z=(r,\theta)} \left\llbracket r^{n^*} \left( \vec{a_j} \cos(n^* \theta) + \vec{b_j} \sin(n^* \theta) \right) \right\rrbracket
\]
with $\vec{a}_j$ and $\vec{b_j}$ linearly independent, or $\tilde{f}_j \equiv \a{0}$. In the first case we must have $0\not\in {\rm spt}\, (\tilde{g}_j (\theta))$ for every $\theta$, so we are necessarily in the second case, where $\tilde{f}_j$ is a single-valued harmonic function. Arguing as above we then conclude that it must coincide with $\tilde{f}_0^+$ on the upper half plane.

\medskip

{\bf Conclusion}.
\eqref{e:non-intersection-1} and \eqref{e:non-intersection-2} lead immediately to the following conclusions:
\begin{itemize}
\item When $Q_0=2$, $f$ takes necessarily the form in (c), where we set $f_0 = \tilde{f}_0$ and $f_j = \tilde{f}_j$;
\item When $Q_0=1$, either the supports of $\tilde{g}_j (\theta)$ are disjoint from those of $\tilde{g}^+_0 (\theta)$ for all $\theta$ and $j$, in which case we set $f_0 = \tilde{f}_0$ and $f_j = \tilde{f}_j$; or there is one $\tilde{f}_j$ which coincides with $\tilde{f}^+_0$ on the upper half plane, while all the others have disjoint supports. This results into a decomposition of the form (a) or (b). If the homogeneity $\alpha$ is an integer, then the decomposition takes the form (a). If the homogeneity $\alpha$ is not integer, then the decomposition takes the form (b) because the only classical harmonic function which is $\alpha$-homogeneous is the trivial one.
\end{itemize}

We come to the final statement of the proposition. If the Dirichlet energy of $f$ is positive, clearly in case (b) the map $\sum_j k_j f_j$ must have nontrivial energy. In case (a), observe that the triviality of $\sum_j k_j f_j$, the Schwarz reflection principle, and assumption \eqref{eq:avgsym} would imply that $f^+ = Q \a{0}$ and $f^- = (Q-1) \a{0}$, which again is not compatible with the positivity of the Dirichlet energy. 
\end{proof}

\section{Proof of Theorem \ref{thm:main}: Discreteness of the singular set}\label{s:final}
The proof of the main theorem is by induction on the number of values $Q$. The basic step $Q=1$ is clearly trivial, because $f^-$ does not exist in that case and $f^+$ is a classical harmonic function. . Now we assume $Q>1$ and, as induction hypotheses, that the theorem holds for every $Q'<Q$.

We argue by contradiction and assume the existence of a Dir-minimizing $\left(Q- \frac12 \right)$-valued planar function with real analytic interface $(\gamma, \varphi)$ whose singular set is not discrete. As shown in Section \ref{s:preliminaries} we can assume, without loss of generality that $f$ is as in Theorem \ref{thm:main_simple}, namely $Q\eta^+=(Q-1) \eta^-$ and the interface is $(\RR, 0)$. Under our assumptions the singular set must have an accumulation point $x_0$. The latter cannot be in the interior, and thus belongs to the interface. Without loss of generality we can assume that $x_0=0$. 

Next, we must have $f^+ (0) = Q \a{0}$. Otherwise we have $f^+ (0)  = Q_1 \a{0} + T$ with $T\in \mathcal{A}_{Q_2} (\R^n)$, where $Q_1+Q_2 = Q$, $1\leq Q_1 \leq Q-1$ and $\supp (T)$ does not contain the origin. By the H\"older continuity theorem, in a neighborhood $U$ of the origin there would be a $Q_2$-valued map $h\in W^{1,2} (U)$ and a $(Q_1-\frac{1}{2})$-valued map $g = (g^+, g^-) \in W^{1,2} (U)$, with disjoint supports and such that $f^\pm = g^\pm + h$. Then the singular set of $f$ in $U$ would be the union of the singular set of $h$ and of the singular set of $f$. Moreover, both must be Dir-minimizing. Hence the singular set of $h$ is discrete by the interior regularity theory, whereas the singular set of $g$ is discrete by the inductive assumption. This is however not possible because we know that $0$ is an accumulation point of the singular set of $f$. 

Note next that it must be $D(r)>0$ for every $r$ in a positive interval, otherwise we would have $f^+ \equiv Q \a{0}$ and $f^- \equiv (Q-1) \a{0}$ in some neighborhood of $0$. Thus $I_{f}(r)$ is well-defined for every $r>0$ sufficiently small. Let $g$ be the (homogeneous) tangent function to $f$ at $0$, given by Theorem \ref{thm:cvtang}. By the characterization in Proposition \ref{prop:chartang} $g$ has the following decomposition:
\[ g^+ = g_0^+ + \sum_{j=1}^J k_j g_j, \quad g^- = g_0^- + \sum_{j=1}^J k_j g_j \]
where:
\begin{itemize}
\item In the alternatives (a) or (b) of Proposition \ref{prop:chartang} $(g_0^+, g_0^-)$ equals $(k_0 \a{h}, (k_0-1) \a{h})$, where $h$ is a classical harmonic function which vanishes at $\RR$ and $g_0^-$ its reflection.
\item In the alternative (c) $(g_0^+, g_0^-) \in W^{1,2}(\RR^2, \mathcal{A}_2^{\pm})$. 
\end{itemize}
In the alternative (b) $g_0^+$ is $2$-valued, namely $g_0^+ = \a{(g_0^+)_1} + \a{(g_0^+)_2}$ and we define
	\[ 
	d_0:= \min_{x\in \mathbb{S}^1_+} {\rm sep} (g_0^+ (x)) = 
	\min_{x\in \mathbb{S}_1^+} \left|(g_0^+)_1 (x) - (g_0^+)_2 (x)\right|\, .
	\]
	Note that $d_0$ is positive. 
In the alternative 	(a) we set $d_0 = +\infty$.

For each $j\in \{1, \cdots, J\}$ we define
\[ 	d_{0,j}:=  \min \left\{ \min_{x\in \mathbb{S}^1_+ } \dist\left(\supp(g_0^+(x)), \supp(g_j(x))\right), \, \min_{x\in \mathbb{S}^1_-} \dist\left(\supp(g_0^-(x)), \supp(g_j(x))\right) \right\},
\]
and define for each pair $i\neq j \in \{1,\cdots, J\}$
\[ 	
	d_{i,j}:= \min_{x\in \So}  \dist\left(\supp(g_i(x)), \supp(g_j(x)\right).
\]
By Proposition \ref{prop:chartang} we know $d_0, d_{0,j}, d_{i,j} > 0$ for all $i, j$, because the Dirichlet energy of the tangent function is positive and it satisfies the averaging condition \eqref{eq:avgsym}.  
Let 
\[ \epsilon = \frac14 \min \left\{d_0, \min_{j} d_{0,j}, \, \min_{i\neq j} d_{i,j} \right\}>0. \]

We claim that there exists $r_0>0$ such that
\begin{equation}\label{e:separation_claim}
\GG(f(x), g(x)) \leq \epsilon|x|^{\alpha} \quad \text{ for every } |x| \leq r_0, 
\end{equation}
where $\alpha = I_{0,f}(0)>0$.
In fact, recall the uniform convergence of the blow-ups $f_r$ to $g$:
\[ \GG(f_r(\theta), g(\theta)) \to 0 \text{ uniformly in } \theta \in \So \text{ as } r\to 0. \]
Recall \eqref{eq:blowupscal}, the blow-ups satisfy
\[ \frac{f(x)}{|x|^\alpha} = f_{|x|} \left( \frac{x}{|x|} \right) \left( 1 + O\left( |x|^{\frac{\beta}{2}} \right) \right). \]
Hence
\[ \GG\left( \frac{f(r,\theta)}{r^\alpha}, g(\theta) \right) \to 0  \text{ uniformly in } \theta \in \So \text{ as } r\to 0. \]
Recall that $g$ is an $\alpha$-homogeneous map, i.e. $g(x) = |x|^\alpha g (\frac{x}{|x|})$. We have thus showed \eqref{e:separation_claim}.

The choice of $\epsilon$ implies the existence of functions $h_j$ with $j\in \{0, 1, \cdots, J\}$, such that:
\begin{itemize}
\item $h_0 = (h^+_0, h_0^-) \in W^{1,2}(B_{r_0}, \mathcal{A}_{Q_0}^{\pm})$ with interface $(\gamma, \varphi)$ and $Q_0 = 1$ or $2$, depending on whether alternative (a) or (b) in Proposition \ref{prop:chartang} holds, and in particular $\card \supp(h_0^+(x)) = Q_0$ for all $x \in B_{r_0}^+ \setminus \{ 0\}$;
\item each $h_j$ is in $ W^{1,2}(B_{r_0}, \mathcal{A}_{k_j Q_j} ) $, and
\begin{equation}\label{eq:decompf}
	f|_{B_{r_0}} = (h_0^+, h_0^-) + \sum_{j=1}^J h_j;
\end{equation} 
\item For every $x\in B_{r_0}\setminus \{ 0\}$ and every $i>j>0$ we have $\supp (h_j (x)) \cap \supp (h_i (x))=\emptyset$;
\item For every $x\in B_{r_0}^+\setminus \{0\}$ and every $i>0$ we have $\supp (h_i (x))\cap \supp (h_0^+ (x)) =\emptyset$;
\item For every $x\in B_{r_0}^-\setminus \{0\}$ and every $i>0$ we have $\supp (h_i (x))\cap \supp (h_0^- (x))=\emptyset$.
\end{itemize}
In particular:
\begin{itemize}
\item $h_0$ is a Dir-minimizer with interface $(\RR, 0)$, and each $h_j$ is a Dir-minimizer;
\item The singular set of $f$ in $B_{r_0}$ is given by $0$ and the union of the singular sets of $h_0^+, h_0^-$ and the $h_j$'s.
\end{itemize} 
Suppose $J=0$. Recall Proposition \ref{prop:chartang}, this may only occur in the alternative (c), i.e. when $f|_{B_{r_0}} = (h_0^+, h_0^-)$ is a $\frac{3}{2}$-valued map. By the separation of sheets of $h_0^+$, the singular set of $f$ in $B_{r_0}$ is just the origin and we get a contradiction.
Suppose $J\geq 1$, in other words the sum \eqref{eq:decompf} contains at least two terms, so $h_0^+$ takes strictly less than $Q$ values and we can use our inductive hypothesis to conclude that the singular set of $h_0$ is discrete. On the other hand, the singular set of each $h_j$ with $j>0$ is discrete by \cite[Theorem 0.12]{DS}. We conclude that the singular set of $f$ in $B_{r_0}$ is discrete as well, contradicting the assumption that the origin was an accumulation point for it.

\section{The exceptional $\frac{2}{3}$-homogeneous minimizer}

In this section we complete the analysis of the singularities by showing the following

\begin{prop}\label{p:exception}
Let $\vec{c} \neq 0$ and let $f_0$ be the $\frac{3}{2}$-valued map of case (ii) in Lemma \ref{l:char_irr_hom}. Then $f_0$ is locally Dir-minimizing.
\end{prop}
\begin{proof}
First of all consider that $f_0$ takes values in the line spanned by $\vec{c}$. Thus it suffices to show the claim under the assumption that $n=1$. 
Consider now the scalar functions
\[
g^+ (\theta) = \a{\sin \frac{2}{3}\theta} + \a{\sin \frac{2}{3} (\theta+2\pi)}\, ,\qquad \theta \in [0, \pi]
\]
\[
g^- (\theta) = \a{\sin \frac{2}{3} \theta}\, , \qquad \theta \in [\pi, 2\pi]\, ,
\]
and the corresponding $\frac{3}{2}$-valued map $(g^+, g^-)$ on $\So$. 
Denote by $h = (h^+, h^-)$ any minimizer of the corresponding Dirichlet problem for $\frac{3}{2}$-valued maps with interface $(\mathbb R, 0)$, which can be shown to exist by the direct methods of the calculus of variations following the theory in \cite{DDHM}. Note that in order to apply the direct methods we need to show the existence of at least one $\frac{3}{2}$-valued function with finite Dirichlet energy which takes the Dirichlet boundary data and has interface $(\mathbb R, 0)$. However such function is provided precisely by the $f_0$ of case (ii) in Lemma \ref{l:char_irr_hom}. 

A simple computation shows that $2 \bdeta (g^+ (\theta)) = g^- (2\pi-\theta)$ for every $\theta$.
In particular, it must be that $2\bdeta (h^+ (x)) = h^- (\bar x)$, otherwise we could  argue as in Section \ref{s:mod_out_average} and lower the energy of $h$ by keeping the same boundary value and the same interface. Hence $h$ satisfies the condition \eqref{eq:avgsym}. 
Now, $h^-$ is a classical harmonic function and by Theorem \ref{thm:Holder} it has continuous trace on the open segment $](-1,0), (1,0)[$. We will show below that there must be one point $\sigma \in ]-1,1[$ such that $h^- (\sigma ,0)=0$. Fix now such $p = (\sigma, 0)$ and observe that the Dirichlet energy of $h$ cannot vanish in any disk $B_r (p)$: if it vanishes on some disk, by the unique continuation of classical harmonic function $h^-$ would have to vanish identically, which is not possible because its trace on $(\So)^-$ is not identically $0$. Consider now the unique tangent function $f$ to $h$ at $p$. The latter must be an $\alpha$-homogeneous Dir-minimizer, satisfying the averaging condition \eqref{eq:avgsym}.
We claim that such tangent function must be necessarily of the form $f = (f_0^+, f_0^-)$ as in case (c) of Proposition \ref{prop:chartang}, which thus would prove the minimality of $f$ (by the compactness of Dir-minimizers). Indeed, if this were not the case, then $f$ would have to fall necessarily in case (a) of Proposition \ref{prop:chartang}, because in order to fall in case (b) the map would have to be $(Q-\frac{1}{2})$-valued with $Q\geq 3$. In case (a) of Proposition \ref{prop:chartang} we would have
\begin{align}
f^+ (x) & = \a{k (x)} + \a{l (x)},\\
f^- (x) & = \a{l (x)},
\end{align}
where both $k$ and $l$ are homogeneous harmonic polynomials with the same degree of homogeneity $d\geq 1$. Moreover $k (s,0) =0$ for every $s$. 
In particular 
\begin{align*}
k (r, \theta) &= a r^d \sin d \theta\\
l (r, \theta) &= r^d (\alpha \sin d\theta + \beta \cos d \theta)
\end{align*}
for some constants $a, \alpha, \beta \in \mathbb R$. Notice however that $k$ and $l$ cannot coincide, because otherwise the averaging condition \eqref{eq:avgsym} would imply that they both vanish identically, whereas the Dirichlet energy of $f$ must be positive. Since $k$ and $l$ do not coincide, Proposition \ref{prop:chartang} implies that 
\[
a \sin d\theta \neq \alpha \sin d\theta + \beta \cos d\theta \qquad \forall \theta \in [0, \pi]\, ,
\]
namely
\[
(\alpha -a) \sin d\theta + \beta \cos d\theta \neq 0 \qquad \forall \theta \in [0, \pi]\, .
\]
The latter condition is however impossible. 

It remains to show the existence of $\sigma\in ]-1,1[$ such that $h^- (\sigma, 0) =0$. As already recalled, $h^-$ has a continuous trace on $]-1,1[$. If we knew the continuity of $h^-$ also at the ``corner points'' $(-1,0)$ and $(1,0)$, then we would have
\begin{align}
h^- (1, 0) & = \sin \frac{4\pi}{3} < 0 \\
h^- (-1, 0) & = \sin \frac{2\pi}{3} > 0
\end{align}
and the existence of the point $\sigma$ would be guaranteed by the intermediate value theorem for continuous functions. While it is possible to show a general continuity result at the intersection of the ``boundary'' $\So$ with the interface $\mathbb R$ under rather general assumptions on the boundary data and for a general $(Q-\frac{1}{2})$-Dir minimizer, this would require quite some effort and goes beyond the scopes of the present paper. We circumvent this technical difficulty with a short ad hoc argument. 

In order to prove that $\sigma$ exists it suffices indeed to argue that $h^-$ must take both positive and negative values on $]-1,1[$. For this it suffices to show the existence of a sequence of values $s_k\uparrow 1$ such that $h^- (s_k, 0) \to \sin \frac{4\pi}{3}$ and of a sequence $t_k \downarrow -1$ such that 
$h^- (t_k, 0) \to \sin \frac{2\pi}{3}$. Without loss of generality, let us argue for the existence of the sequence $s_k$. Assume by contradiction that there are a positive $\delta$ and positive $\eta$ such that
\[
\left|h^- (1-t, 0) - \sin \frac{4\pi}{3}\right|\geq 2\delta \qquad \forall t\in ]0, \eta[\, .
\]
Let $\gamma_t$ be the arc $\partial B_t (1,0)\cap B_1 (0) \cap \{(x,y): y<0\}$. One endpoint $p_t$ of $\gamma_t$ is precisely $(1-t, 0)$, while the other endpoint $q_t$ lies on $(\So)^-$. By choosing $\eta$ sufficiently small we can use the continuity of the harmonic function $h^-$ at $q_t$ and the continuity of its trace $g^-$ at $(1,0)$ to infer
\[
\left|h^- (q_t) - \sin \frac{4\pi}{3}\right| = |g^- (q_t) - g^- (1,0)| \leq \delta\, .
\]
We have thus concluded that
\[
|h^- (p_t) - h^- (q_t)|\geq \delta\, . 
\]
By the fundamental theorem of calculus and using Cauchy-Schwarz, for every $\rho \in ]0, \eta[$ such that $h^-|_{\gamma_\rho} \in W^{1,2} (\gamma_\rho)$ we conclude
\[
\int_{\gamma_\rho} |D_{\tau} h^-|^2 \geq \frac{C}{\rho} \left(\int_{\gamma_\rho} |D_{\tau} h^-|\right)^2 \geq \frac{C\delta^2}{\rho}\, , 
\]
where $C$ is a geometric constant. The latter inequality thus holds for a.e. $\rho\in ]0, \eta[$ and integrating in $\rho$ we then conclude
\[
\iint_{B_\eta (1,0)} |Dh^-|^2 \geq \int_0^{\eta} \left( \int_{\gamma_{\rho}} |D_{\tau} h^-|^2 \right) d\rho \geq \int_0^\eta \frac{C\delta^2}{\rho}\, d\rho = \infty\, ,
\]
which contradicts the fact that $h^-$ has finite energy. 
\end{proof}

\bibliographystyle{abbrv}
\bibliography{white-conjecture-linear}

\begin{thebibliography}{10}

\bibitem{AllB}
W.~K. Allard.
\newblock {On the first variation of a varifold: boundary behavior}.
\newblock {\em Ann. of Math. (2)}, 101:418--446, 1975.

\bibitem{Alm}
J.~F.~J. Almgren.
\newblock {\em {Almgren's big regularity paper}}, volume~1 of {\em {World
  Scientific Monograph Series in Mathematics}}.
\newblock World Scientific Publishing Co. Inc., River Edge, NJ, 2000.

\bibitem{Chang}
S.~X. Chang.
\newblock {Two-dimensional area minimizing integral currents are classical
  minimal surfaces}.
\newblock {\em J. Amer. Math. Soc.}, 1(4):699--778, 1988.

\bibitem{Courant40}
R.~Courant.
\newblock {The existence of minimal surfaces of given topological structure
  under prescribed boundary conditions}.
\newblock {\em Acta Math.}, 72:51--98, 1940.

\bibitem{DG}
E.~{De Giorgi}.
\newblock {\em {Frontiere orientate di misura minima}}.
\newblock {Seminario di Matematica della Scuola Normale Superiore di Pisa,
  1960-61}. Editrice Tecnico Scientifica, Pisa, 1961.

\bibitem{De-GiorgiColombiniPiccinini72}
E.~{De Giorgi}, F.~Colombini, and L.~C. Piccinini.
\newblock {\em {Frontiere orientate di misura minima e questioni collegate}}.
\newblock Scuola Normale Superiore, Pisa, 1972.

\bibitem{DDH}
C.~{De Lellis}, G.~{De Philippis}, and J.~{Hirsch}.
\newblock {Forthcoming}.

\bibitem{DDHM}
C.~{De Lellis}, G.~{De Philippis}, J.~{Hirsch}, and A.~{Massaccesi}.
\newblock {On the boundary behavior of mass-minimizing integral currents}.
\newblock {\em arXiv e-prints}, page arXiv:1809.09457, Sep 2018.

\bibitem{DHMS1}
C.~{De Lellis}, J.~{Hirsch}, A.~{Marchese}, and S.~{Stuvard}.
\newblock {Forthcoming}.

\bibitem{DHMS2}
C.~{De Lellis}, J.~{Hirsch}, A.~{Marchese}, and S.~{Stuvard}.
\newblock {Forthcoming}.

\bibitem{DSS2}
C.~{De Lellis}, E.~{Spadaro}, and L.~{Spolaor}.
\newblock {Regularity theory for $2$-dimensional almost minimal currents I:
  Lipschitz approximation}.
\newblock {\em ArXiv e-prints. To appear in {Trans. Amer. Math. Soc.}}, Aug.
  2015.

\bibitem{DSS4}
C.~{De Lellis}, E.~{Spadaro}, and L.~{Spolaor}.
\newblock {Regularity theory for $2$-dimensional almost minimal currents III:
  blowup}.
\newblock {\em ArXiv e-prints. To appear in {Jour. of Diff. Geom}}, Aug. 2015.

\bibitem{DSS3}
C.~{De Lellis}, E.~Spadaro, and L.~Spolaor.
\newblock {Regularity {T}heory for 2-{D}imensional {A}lmost {M}inimal
  {C}urrents {II}: {B}ranched {C}enter {M}anifold}.
\newblock {\em Ann. PDE}, 3(2):3:18, 2017.

\bibitem{DSS1}
C.~{De Lellis}, E.~Spadaro, and L.~Spolaor.
\newblock {Uniqueness of tangent cones for two-dimensional almost-minimizing
  currents}.
\newblock {\em Comm. Pure Appl. Math.}, 70(7):1402--1421, 2017.

\bibitem{DS}
C.~De~Lellis and E.~N. Spadaro.
\newblock {$Q$}-valued functions revisited.
\newblock {\em Mem. Amer. Math. Soc.}, 211(991):vi+79, 2011.

\bibitem{Douglas}
J.~Douglas.
\newblock {Minimal surfaces of higher topological structure}.
\newblock {\em Ann. of Math. (2)}, 40(1):205--298, 1939.

\bibitem{FF}
H.~Federer and W.~H. Fleming.
\newblock {Normal and integral currents}.
\newblock {\em Ann. of Math. (2)}, 72:458--520, 1960.

\bibitem{Fleming}
W.~H. Fleming.
\newblock {An Example in the Problem of Least Area}.
\newblock {\em P. Am. Math. Soc.}, 7:1063--1074, 1956.

\bibitem{HS}
R.~Hardt and L.~Simon.
\newblock {Boundary regularity and embedded solutions for the oriented
  {P}lateau problem}.
\newblock {\em Ann. of Math. (2)}, 110(3):439--486, 1979.

\bibitem{Spolaor}
L.~{Spolaor}.
\newblock {Almgren's type regularity for {S}emicalibrated {C}urrents}.
\newblock {\em ArXiv e-prints}, Nov. 2015.

\bibitem{White97}
B.~White.
\newblock {Classical area minimizing surfaces with real-analytic boundaries}.
\newblock {\em Acta Math.}, 179(2):295--305, 1997.

\end{thebibliography}

\end{document}